\documentclass[11pt]{article}
\usepackage[colorlinks]{hyperref}
\usepackage{color}
\usepackage{graphicx}
\usepackage{graphics}
\usepackage{makeidx}
\usepackage{showidx}
\usepackage{latexsym}
\usepackage{amssymb}
\usepackage{verbatim}
\usepackage{amsmath}
\usepackage{amsthm}
\usepackage{amsfonts}
\usepackage{subcaption} 
\usepackage[abs]{overpic}
\usepackage{xcolor,varwidth}
\usepackage{geometry}
\newtheorem{theorem}{Theorem}[section]
\newtheorem{proposition}[theorem]{Proposition}
\newtheorem{remark}[theorem]{Remark}
\newtheorem{claim}[theorem]{Claim}
\newtheorem{question}[theorem]{Question}
\newtheorem{definition}[theorem]{Definition}
\newtheorem{corollary}[theorem]{Corollary}
\newtheorem{lemma}[theorem]{Lemma}
\newtheorem{example}[theorem]{Example}
\newtheorem{conjecture}[theorem]{Conjecture}

\title{Symplectic Geometry of Anosov Flows in Dimension 3 and Bi-Contact Topology}
\author{Surena Hozoori}
\newcommand{\Addresses}{{
  \bigskip
  \footnotesize

Surena Hozoori, \textsc{Department of Mathematics, Georgia Institute of Technology,
    Atlanta, GA 30332}\par\nopagebreak
  \textit{E-mail address}: \texttt{shozoori3@gatech.edu}
  
  }}
  \date{}
\begin{document}
\maketitle
\noindent
\begin{abstract}
We give a purely contact and symplectic geometric characterization of Anosov flows in dimension 3 and discuss a framework to use tools from contact and symplectic geometry and topology in the study of Anosov dynamics. We also discuss some uniqueness results regarding the underlying (bi)-contact structures for an Anosov flow and give a characterization of Anosovity based on Reeb flows.
\end{abstract}

\section{Introduction}\label{1}

Anosov flows were introduced by Dimitri Anosov \cite{anosov0,anosov} in 1960s as a generalization of geodesic flows of hyperbolic manifolds and were immediately considered an important class of dynamical systems, thanks to many interesting global properties. Many tools of dynamical system, including ergodic theory, helped to increase our understanding of Anosov flows (see \cite{marg} for early developments). But more profound connections to the topology of the underlying manifold, were discovered in dimension 3, thanks to the use of foliation theory. This was initiated by many, including Thurston, Plante and Verjovsky.  However, more recent advances in the mid 1990s came from new techniques in foliation theory, introduced by Sergio Fenley, alongside Barbot, Barthelmé, etc (see \cite{fenley1} as the seminal work and \cite{topanosov} for a nice survey of such results).

The goal of this paper is to discuss a geometric and topological framework for the study of Anosov flows in dimension 3, provided thanks to our purely contact and symplectic characterization of such~flows.

In this paper, we consider $M$ to be a closed, connected, oriented 3-manifold and assume (projectively) Anosov flows to be {\em orientable}, i.e. the associated {\em stable and unstable directions} are orientable line fields (Assuming the orientability of $M$, this can be achieved, possibly after going to a double cover of $M$). See Section~\ref{2} and Section~\ref{3} for related definitions and discussions.

\begin{theorem}\label{main1}

Let $\phi^t$ be a flow on the 3-manifold $M$, generated by the $C^1$ vector field $X$. Then $\phi^t$ is Anosov, if and only if, $\langle X \rangle= \xi_+ \cap \xi_-$, where $\xi_+$ and $\xi_-$ are transverse positive and negative contact structures, respectively, and  there exist contact forms $\alpha_+$ and $\alpha_-$ for $\xi_+$ and $\xi_-$, respectively, such that $(\alpha_-,\alpha_+)$ and $(-\alpha_-,\alpha_+)$ are Liouville pairs.

\end{theorem}

Although the relation to contact geometry was observed by Eliashberg-Thurston \cite{confoliations} and more thoroughly by Mitsumatsu \cite{mitsumatsu}, we improve those observations into a full characterization of such flows. More precisely, Mitsumatsu \cite{mitsumatsu} proves the existence of Liouville pairs in the case of {\em smooth volume preserving} Anosov flows, in order to introduce a large family of {\em non-Stein Liouville domains} (generalizing a previous work of McDuff \cite{stein}). Nevertheless, our goal in this paper is to give a complete characterization of Anosov flows in full generality, which is necessary for developing a contact and symplectic geometric theory of such flows. In order to do so, we use natural geometric quantities, namely {\em expansion rates} (see Section~\ref{3}), to achieve a refinement of Mitsumatsu's observation, which helps us generalize his result to an arbitrary (possibly non-volume preserving) $C^1$ Anosov flow, as well as prove the converse. Note that Brunella \cite{brunella} has shown the abundance of Anosov flows with no invariant volume forms. However,  the main technical difficulty is that for smooth volume preserving Anosov flows, the weak stable and unstable bundles are known to be at least $C^1$ \cite{hruder}, which significantly simplifies the geometry of an Anosov flow (see \cite{mitsumatsu} or Section~5.2 of \cite{hoz3} for this simplified setting). In the absence of such regularity condition in the general setting of Theorem~\ref{main1}, we introduce approximation techniques (see Section~\ref{4}), tailored for the setting of Anosov flows, which facilitate translating the information of a given $C^1$ Anosov flow (with possibly $C^0$ invariant foliations) to the corresponding contact structures (which are at least $C^1$). We believe that these approximation techniques are independently interesting and can be used regarding other questions in Anosov dynamics (for instance, see \cite{hoz3} for application in Anosov surgeries). See Remark~\ref{reg}.


By Theorem~\ref{main1}, the vector field which generates an Anosov flow in particular lies in the intersection of a pair of positive and negative contact structures, i.e. a {\em bi-contact structure}. Eliashberg-Thurston \cite{confoliations} and Mitsumatsu \cite{mitsumatsu} show that this condition in fact has dynamical interpretation and defines a large class of flows, named {\em projectively Anosov flows}. These are flows, which induce, via $\pi:TM\rightarrow TM/\langle X\rangle$, a flow with {\em dominated splitting} on $TM/\langle X\rangle$ (see Section~\ref{3}). 
From this viewpoint, Theorem~\ref{main1} should be seen as a significant extension of the dictionary between Anosov dynamical and symplectic geometric concepts, based on the equivalence between projectively Anosov flows and bi-contact structures, as observed by Eliashberg-Thurston and Mitsumatsu.

We remark that projectively Anosov flows are previously studied in various contexts, under different names. In the geometry and topology literature, beside projectively Anosov flows, they are referred to as {\em conformally Anosov flows} and are studied from the perspectives of foliation theory~\cite{confoliations,noda,asaoka,colin}, Riemannian geometry of contact structures \cite{blperr,perr,hoz2} and Reeb dynamics \cite{hoz}. This is while, in the dynamical systems literature, the term conformally Anosov is preserved for another dynamical concept (for instance see \cite{kanai,sullivan,llave}) and the dynamical aspects of projectively Anosov flows are studied under the titles {\em flows with dominated splitting} (see \cite{beyond,pujals,impa,survdom,clinic}) or {\em eventually relatively pseudo hyperbolic flows} \cite{hps}.

Although, it is not immediately clear if the class of projectively Anosov flows is larger than Anosov flows, first examples of such flows on $\mathbb{T}^3$ and Nil manifolds \cite{mitsumatsu,confoliations}, which do not admit any Anosov flows \cite{fun}, as well as more recent examples of projectively Anosov flows on atoroidal manifolds, which cannot be deformed to Anosov flows \cite{bbp}, proved the properness of the inclusion. In fact, we now know that unlike Anosov flows, projectively Anosov flows are abundant. For instance, there are infinitely many distinct projectively Anosov flows on $\mathbb{S}^3$ and no Anosov flows \cite{otbi}. Therefore, Theorem~\ref{main1} can be seen as a host of geometric and topological rigidity conditions on a projectively Anosov flow. In particular, this enables us to use various contact and symplectic geometric and topological tools in the study of Anosov dynamics. For instance, there are many questions about the knot theory of the periodic orbits of Anosov flows. Thanks to Theorem~\ref{main1}, such periodic orbits are now {\em Legendrian knots} for both underlying contact structures and moreover, correspond to {\em exact Lagrangians} in the constructed Liouville pairs. These are standard and well studied objects in contact and symplectic topology and now, the same techniques can be employed for understanding the periodic orbits of such flows (see Remark~\ref{per}).

In contact topology, thanks to {\em Darboux theorem}, there are no local invariants and {\em Gray's theorem} implies that homotopy through contact structures can be done by an isotopy of the ambient manifold. Therefore, the local structure of contact structures does not carry any information and the subtlety of these structures is hidden in their global topological properties. In fact, we have a hierarchy of topological rigidity conditions on a contact manifold (see Section~\ref{2}). Although it is not trivial, it is known that all the inclusions below are proper.

$$\bigg\{\substack{\text{Stein fillable} \\  \text{contact manifolds}} \bigg\} \mbox{\Large$\subset$}
\bigg\{\substack{\text{Exactly symplectically fillable} \\  \text{contact manifolds}} \bigg\} \mbox{\Large$\subset$}
\bigg\{\substack{\text{Strongly symplectically fillable} \\  \text{contact manifolds}} \bigg\}$$ $$ \mbox{\Large$\subset$}
\bigg\{\substack{\text{Weakly symplectically fillable} \\  \text{contact manifolds}} \bigg\} \mbox{\Large$\subset$}
\bigg\{\substack{\text{Tight} \\  \text{contact manifolds}} \bigg\} \mbox{\Large$\subset$}
\bigg\{ \text{contact manifolds} \bigg\}.$$

Now, we can naturally apply the hierarchy of contact topology to bi-contact structures and therefore, achieve a filtration of Anosovity concepts (see Section~\ref{7} for the precise definitions).

$$\bigg\{\substack{\text{Anosov flows}} \bigg\} \mbox{\Large$\subseteq$}
\bigg\{\substack{\text{Exactly symplectically bi-fillable} \\  \text{projectively Anosov flows}} \bigg\} \mbox{\Large$\subseteq$}
\bigg\{\substack{\text{Strongly symplectically bi-fillable} \\  \text{projectively Anosov flows}} \bigg\}$$ $$ \mbox{\Large$\subseteq$}
\bigg\{\substack{\text{Weakly symplectically bi-fillable} \\  \text{projectively Anosov flows}} \bigg\} \mbox{\Large$\subseteq$}
\bigg\{\substack{\text{Tight} \\  \text{projectively Anosov flows}} \bigg\} \mbox{\Large$\subset$}
\bigg\{ \text{projectively Anosov flows} \bigg\}.$$

In the above hierarchy, we first notice that there are no equivalent of {\em Stein fillable} contact manifolds for bi-contact structures (and projectively Anosov flows), since Stein fillings can only have connected boundaries \cite{stein}. 

The above hierarchy provokes a general line of questioning, which can help us understand Anosov dynamics, through the lens of contact and symplectic topology.

\begin{question}\label{qcontact}
What does each bi-contact topological layer imply about the dynamics of the corresponding class of projectively Anosov flows? What bi-contact topological layer is responsible for a given property of Anosov flows?
\end{question}

One important motivation to study the consequences of these contact geometric conditions on the dynamics of a projectively Anosov flow is the classification of Anosov flows, up to {\em orbit equivalence} (that is up to homemorphisms, mapping the orbits of one flow to another). In the light of Theorem~\ref{main1}, this can be split into two problems: topological classification of bi-contact structures (a contact topological problem) and understanding the bifurcations of projectively Anosov flows, under {\em bi-contact homotopy} (a dynamical question). Since many contact topological tools have been successfully developed in the past few decades to address the first problem, understanding the bifurcation problem can lead to important classification results. A very useful perspective is to notice that all the contact topological properties of the above hierarchy are preserved under bi-contact homotopy and therefore, are satisfied for any projectively Anosov flow which is homotopic to some Anosov flow (see Section~\ref{7} for related discussions).

Regarding Question~\ref{qcontact}, \cite{nos3} shows that there are no tight projectively Anosov flows on $\mathbb{S}^3$ (generalizing non-existence of Anosov flows) and \cite{otbi} gives a partial classification of {\em overtwisted projectively Anosov flows}, i.e. when both contact structures, forming the underlying bi-contact structure, are not tight (are {\em overtwisted}). More precisely, they show that overtwisted projectively Anosov flows exist, when there are no algebraic obstruction. Although, this is not a full classification, it is worth comparing this with purely algebraic classification of overtwisted contact structures, by Eliashberg \cite{elover,elmart}, reaffirming the parallels in the two theories. This also implies that the class of tight projectively Anosov flows is (considerably) smaller than general projectively Anosov flows. 

 We will observe the properness of the middle inclusion, by constructing examples on $\mathbb{T}^3$, while the properness of other inclusions remain unknown (Question~\ref{qprop}).

\begin{theorem}\label{proper1}
There are (trivially) weakly symplectically bi-fillable projectively Anosov flows, which are not strongly symplectically bi-fillable.
\end{theorem}

As mentioned above, all the above contact topological conditions on projectively Anosov flows are purely topological, that is do not depend on the homotopy of any of the two underlying contact structures, with the exception of the first inclusion, i.e. Anosovity of a flow. It turns out that in the study of bi-contact structures (or equivalently, projectively Anosov flows), the local geometry is more subtle than contact structures, due to lack of theorems equivalent to the Darboux and Gray theorems. Drawing contrast between two notions of {\em bi-contact homotopy vs. isotopy},  we conclude that bi-contact homotopy is the natural notion from dynamical point of view (see Definition~\ref{bilocal} and the subsequent discussion). The relation between Anosovity and geometry of bi-contact structures is not well understood and we bring related discussions and questions in Section~\ref{7}.

\begin{question}
How does the Anosovity of a flow depend on the geometry of the underlying bi-contact structure, under bi-contact homotopy?
\end{question}

We show that at least for a fixed projectively Anosov flow, there is a unique {\em supporting} bi-contact structure, up to bi-contact homotopy.

\begin{theorem}\label{uniq1}
If $(\xi_-,\xi_+)$ and $(\xi_-',\xi_+')$ are two supporting bi-contact structures for a projectively Anosov flow, then they are homotopic through supporting bi-contact structures.
\end{theorem}

Furthermore, for an Anosov flow, the construction of Liouville pairs in Theorem~\ref{main1} can be based on an arbitrary choice of bi-contact structure.

\begin{theorem}\label{main11}
Suppose $\phi^t$ is an Anosov flow on a 3-manifold $M$, generated by a $C^1$ vector field $X$ and let $(\xi_-,\xi_+)$ be any supporting bi-contact structure for $X$. Then, for any choice of orientations for $\xi_-$ and $\xi_+$, there exist contact forms $\alpha_-$ and $\alpha_+$, such that $\ker{\alpha_-}=\xi_-$ and $\ker{\alpha_+}=\xi_+$, respecting the orientations, and $(\alpha_-,\alpha_+)$ is a Liouville pair.
\end{theorem}

We also use well known facts in Anosov dynamics, as well as the underlying techniques of Theorem~\ref{uniq1}, to derive a family of uniqueness results for the underlying contact structures, reducing the study of the supporting bi-contact structure to only one of the supporting contact structures.

\begin{theorem}
If $M$ is atoroidal and $(\xi_-,\xi_+)$ a supporting bi-contact structure for the Anosov vector field $X$ on $M$, then for any supporting positive contact structure $\xi$, $\xi$ is isotopic to $\xi_+$, through supporting contact structures.
\end{theorem}

\begin{theorem}
Let $X$ be a skewed $\mathbb{R}$-covered Anosov vector field, supported by the bi-contact structure $(\xi_-,\xi_+)$ on $M$, and let $\xi$ be any supporting positive contact structure. Then $\xi$ is isotopic to $\xi_+$, through supporting contact structures.
\end{theorem}

\begin{theorem}
Let $X$ be the suspension of an Anosov diffeomorphism of torus, supported by the bi-contact structure $(\xi_-,\xi_+)$, and $\xi$ a positive supporting contact structure. Then, $\xi$ is isotopic through supporting bi-contact structures to $\xi_+$, if and only if, $\xi$ is strongly symplectically fillable.
\end{theorem}

On a separate note, we also use the ideas developed in Section~\ref{3} and proof of Theorem~\ref{main1} to give a characterization of Anosovity, based on the {\em Reeb vector fields}, associated to the underlying contact structures. Reeb vector fields play a very important role in contact geometry and Hamiltonian mechanics and since early 90s, their deep relation to the topology of contact manifolds has been explored.

\begin{theorem}\label{reeb1}
Let $X$ be a projectively Anosov vector field on $M$. Then, the followings are equivalent:

(1) $X$ is Anosov;

(2) There exists a supporting bi-contact structure $(\xi_-,\xi_+)$, such that $\xi_+$ admits a Reeb vector field, which is dynamically negative everywhere;

(3) There exists a supporting bi-contact structure $(\xi_-,\xi_+)$, such that $\xi_-$ admits a Reeb vector field, which is dynamically positive everywhere.
\end{theorem}

We use the above characterization to show:

\begin{theorem}
Let $(\xi_-,\xi_+)$ be a supporting bi-contact structure for an Anosov flow. Then, $\xi_-$ and $\xi_+$ are {\em hypertight}. That is, they admit contact forms, whose associated Reeb flows do not have any contractible periodic orbit.
\end{theorem}

Consequently, the classical results of Hofer, et al in Reeb dynamics \cite{hofermain,hwz,wendl} would imply the followings:

\begin{corollary}\label{correeb}
Let $(\xi_-,\xi_+)$ be a supporting bi-contact structure for an Anosov flow. Then,

(1) $\xi_-$ and $\xi_+$ are universally tight;

(2) $M$ is irreducible;

(3) there are no exact symplectic cobordisms from $(M,\xi_+)$ or $(-M,\xi_-)$ to $(\mathbb{S}^3,\xi_{std})$.
\end{corollary}

We note that (1) in Corollary~\ref{correeb} is also concluded from Theorem~\ref{main1}, as well as the perturbation theory of $C^0$ foliations \cite{bowden,kazez}, and (2) is a classical fact from Anosov dynamics, and we are giving new Reeb dynamical proofs for them. On the other hand, part (3) is a symplectic improvement of the main result of ~\cite{nos3} and non-existence of Anosov flows on $\mathbb{S}^3$.

\vskip0.25cm

In Section~\ref{2}, we review some background from contact and symplectic topology, which provides context for this paper. In Section~\ref{3}, we discuss (projective) Anosovity of flows, with emphasis on the geometry of the expansion in the stable and unstable direction, which lies in the heart of Theorem~\ref{main1}. We prove Theorem~\ref{main1} and \ref{main11} in Section~\ref{4}. Section~\ref{5} is devoted to various uniqueness theorems for the underlying (bi)-contact structures of (projectively) Anosov flows. In Section~\ref{6}, we prove Theorem~\ref{reeb1}, as a consequence of ideas developed in Section~\ref{3} and the proof of Theorem~\ref{main1}, as well as discuss its contact topological consequences. Finally, Section~\ref{7} is devoted to setting up the contact topological framework for the systematic use of contact and symplectic topological methods in Anosov dynamics. The necessary remarks and definitions are discussed there, as well as a handful of related open problems and conjectures. A proof of Theorem~\ref{proper1} is also given in this section.

\vskip0.5cm
\textbf{ACKNOWLEDGEMENT:} I want to thank my advisor, John Etnyre, for continuous support and discussions along the way, as well as suggesting Example~\ref{exbicontact} and helping with the details of the proof of Theorem~\ref{john}. I am grateful to Thomas Massoni for pointing out a mistake in an earlier version of Theorem~\ref{main11}. I am also very thankful to Rafael de la Llave, Gabriel Paternain, Jonathan Bowden, Thomas Barthelmé and Albert Fathi for helpful discussions and suggested revisions. The author was partially supported by the NSF grants DMS-1608684 and DMS-1906414.

\section{Background On Contact And Symplectic Topology}\label{2}

In this section, we review some basic notions from contact and symplectic geometry and topology, which will be useful in the rest of the paper. We refer the reader to \cite{geiges} for more on these topics.

\begin{definition}
We call a $C^1$ 1-form $\alpha$ a {\em contact form} on $M$, if
$\alpha \wedge d\alpha$ is a non-vanishing volume form on $M$.
If $\alpha \wedge d\alpha>0$ (compared to the orientation on $M$), we call $\alpha$ a {\em positive} contact form and otherwise, a {\em negative} one. We call $\xi:=\ker{\alpha}$ a (positive or negative) {\em contact structure} on $M$. Moreover, we call the pair $(M,\xi)$ a {\em contact manifold}. When not mentioned, we assume the contact structures to be positive.
\end{definition}

Recall that by Frobenius theorem, the contact structure $\xi$ in the above definition is a $C^1$ {\em coorientable maximally non-integrable} plane field on $M$.

\begin{example}\label{excontact}
Some examples of contact structures are:

1) The 1-form $\alpha_{std}=dz-y\, dx$ is a positive contact form on $\mathbb{R}^3$. We call $\xi_{std}=ker{\alpha_{std}}$ the {\em standard} positive contact structure on $\mathbb{R}^3$. Similarly, $\ker{(dz+y\, dx)}$ is the standard negative contact structure on $\mathbb{R}^3$.

2) Consider $\mathbb{C}^2$ equipped with $J$, the standard complex structure on $T\mathbb{C}^2$ and let $\mathbb{S}^3$ be the unit sphere in $\mathbb{C}^2$. It can be seen that the plane field $\xi_{std}:=T\mathbb{S}^3 \cap J T\mathbb{S}^3$ is a contact structure on $\mathbb{S}^3$, referred to as {\em standard} contact structure on $\mathbb{S}^3$. Alternatively, $\xi_{std}$ can be defined as the unique complex line tangent to the unit sphere. It is helpful top note that this contact structure is the one point compactification of the standard contact structure on $\mathbb{R}^3$. Similarly, we can construct a negative contact structure on $\mathbb{S}^3$, by considering the conjugate of the complex structure $J$.

3) Consider $\mathbb{T}^3\simeq \mathbb{R}^3/\mathbb{Z}^3$. It can be seem that the plane fields $\xi_n=\ker{\{\cos{2\pi nz}dx-\sin{2\pi nz}dy\}}$ are positive and negative contact structures on $\mathbb{T}^3$, for integers $n>0$ and $n<0$, respectively.
\end{example}
 
{\em Gray's theorem} states that members of any $C^1$-family of contact structures are isotopic as contact structures and according to {\em Darboux theorem}, all contact structures locally look the same. i.e. around each point in a contact manifold $(M,\xi)$, there exists a neighborhood $U$ and a diffeomorphism of $U$ to $\mathbb{R}^3$, mapping $\xi$ to the standard contact structure (positive or negative one, depending on whether $\xi$ is positive or negative) on $\mathbb{R}^3$. While this means that contact structures lack local invariants, it turns out understanding their topological properties is more subtle and interesting. The most important global feature of contact structures is {\em tightness}, introduced by Eliashberg \cite{elover}, and determining whether a given contact is tight, as well as classifying such contact manifolds, are a prominent theme in contact topology.

 \begin{definition}
 The contact manifold $(M,\xi)$ is called {\em overtwisted}, if $M$ contains an embedded disk that is tangent to $\xi$ along its boundary. Otherwise, $(M,\xi)$ is called {\em tight}. Moreover, $\xi$ universally tight, if its lift to the universal cover of $M$ is tight as well.
 \end{definition}
 
 The significance of the above dichotomy is the classification of overtwisted contact structures by Y. Eliashberg \cite{elmart,elover}. He showed overtwisted contact structures, up to isotopy, are in one to one correspondence with plane fields, up to homotopy (in particular, they always exist). This means overtwisted contact structures do not carry more topological information than plane fields. On the other hand, tight contact structures reveal deeper information about their underlying manifold, and are harder to find, understand and classify.

\begin{remark}
It can be shown that all the contact structures in Example~\ref{excontact} are (universally) tight. In fact, they are the only tight contact structures, up to isotopy, on their underlying manifolds. Note, that all those manifolds admit overtwisted contact structures as well.
\end{remark}

It turns out that one can determine tightness of a contact manifold is based on its relation to four dimensional {\em symplectic topology}, the even dimensional sibling of contact topology.

\begin{definition}
Let $X$ be an oriented 4-manifold. We call a 2-form $\omega$ on $X$ a symplectic form, if it is closed and $\omega \wedge \omega>0$. The pair $(X,\omega)$ is called a symplectic manifold.
\end{definition}

The 2-form form $\omega_{std}=d\left( x_1dy_1+x_2dy_2\right)=dx_1\wedge dy_1+dx_2\wedge dy_2$ is a symplectic form defined on $\mathbb{R}^4$ with coordinates $(x_1,y_1,x_2,y_2)$ (known as the {\em standard} symplectic form), and the Darboux theorem in symplectic geometry states that all symplectic structures are locally equivalent, up to {\em symplectic deformation}. Using the theory of $J$-holomorphic curves, Gromov and Eliashberg proved \cite{grfilltight,elfilltight} that a contact structure is tight, when it is {\em symplectically fillable}, even in the weakest sense.

\begin{definition}
Let $(M,\xi)$ be a contact manifold. We call the symplectic manifold $(X,\omega)$ a {\em weak symplectic filling} for $(M,\xi)$, if $\partial X=M$ as oriented manifolds and $\omega |_\xi>0$. We call $(X,\omega)$ a {\em strong symplectic filling}, if moreover, $\omega=d\alpha$ in a neighborhood of $M=\partial X$, for some 1-form $\alpha$, such that $\alpha|_{TM}$ is a contact form for $\xi$. Finally, we call $(X,\omega)$ an {\em exact symplectic filling} for $(M,\xi)$, if such 1-form $\alpha$ can be defined on all of $X$. We call such $(M,\xi)$ {\em (weakly, strongly or exactly) symplectically fillable}.
\end{definition}

\begin{theorem}\label{gromov}(Gromov 85 \cite{grfilltight}, Eliashberg 90 \cite{elfilltight})
If $(M,\xi)$ is (weakly, strongly or exactly) symplectically fillable, then it is tight.
\end{theorem}

\begin{remark}\label{regrem}
We note that for a 2-form $\omega$ to be symplectic, it needs to be at least $C^1$, because of the closedness condition. However, when $\omega$ is exact, i.e. $\omega=d\alpha$ for some 1-form $\alpha$, this condition is automatically satisfied, assuming the required regularity. So for most purposes, we don't need to assume, for an exact 2-form $\omega=d\alpha$, any regularity more than $C^0$, and methods of symplectic geometry and topology, in particular, the use of $J$-holomorphic curves and Theorem~\ref{gromov}, can be applied. We can also approximate such $\omega=d\alpha$, by symplectic forms of arbitrary high regularity, using $C^1$-approximations of $\alpha$.  Therefore, we still call such $\omega$ symplectic, especially in Theorem~\ref{main}.
\end{remark}

It is known that not all tight contact structures are weakly symplectically fillable, the set of strongly symplectically fillable contact manifolds is a proper subset of the set of weakly symplectically fillable contact manifolds and the set of exactly symplectically fillable contact manifolds is a proper subset of the set of strongly symplectically fillable contact manifolds. Moreover, if a disconnected contact manifold is strongly or weakly symplectically fillable, each of its components is strongly or weakly symplectically fillable, respectively \cite{elfilling,etfilling}.

\begin{example}\label{exfilling}
1) The unit ball in $(\mathbb{R}^4,\omega_{std})$ is a strong symplectic filling for $(\mathbb{S}^3,\xi_{std})$, considered as the unit sphere in $\mathbb{R}^4$.

2) We want to show that all tight contact structures on $\mathbb{T}^3$, given in Example~\ref{excontact}~3), are weakly symplectically fillable. We can observe that after an isotopy $\xi_n=\ker{dz+\epsilon\{\cos{2\pi nz}dx-\sin{2\pi nz}dy\}}$ for small $\epsilon>0$, i.e. we can isotope $\xi_n$ to be arbitrary close to the horizontal foliation $\ker{dz}$ on $\mathbb{T}^3$. Now consider the symplectic manifold $(X,\omega)=(\mathbb{T}^2\times D^2,\omega_1\oplus \omega_2)$, where $\omega_1$ and $\omega_2$ are area forms for $T^2$ and $D^2$, respectively. Clearly, $\partial X=\mathbb{T}^3$ and if at the boundary, we consider the coordinates $(x,y)$ for $\mathbb{T}^2$ and $z$ for the angular coordinate of $D^2$, we have $\omega|_{\ker{dz}}>0$. Since, for small $\epsilon>0$, $\xi_n$ is a small perturbation of $\ker{dz}$, we also have $\omega|_{\xi_n}>0$. Therefore, all $\xi_n$s are weakly symplectically fillable. It can be seen \cite{eltorus} that except $\xi_1$, none of these contact structures, are strongly symplectically fillable and the canonical symplectic structure on the cotangent bundle $T^*T^2$, provides an exact symplectic filling for $(T^3,\xi_1)$.
\end{example}

\begin{remark}\label{torsion}
The concept of {\em Giroux torsion} was introduced by Emmanuel Giroux \cite{giroux}. A contact manifold $(M,\xi)$ is said to contain Giroux torsion, if it admits a contact embedding of $$\left( [0,2\pi]\times \mathbb{S}^1\times\mathbb{S}^1 \text{ with coordinates }(t,\phi_1,\phi_2),\ \ker{(\cos{t}\,d\phi_1+sin{t}\,d\phi_2}) \right)\rightarrow(M,\xi).$$

Note that all the tight contact structures on $\mathbb{T}^3$, discussed in Example~\ref{excontact}~3) contain Giroux torsion, except for $n=1$. Later in \cite{gay}, it was proven that contact structures containing Giroux torsion do not admit strong symplectic fillings. This notion can be generalized by considering {\em Giroux $\pi$-torsion}. i.e. when the contact manifold contains half of a Giroux torsion: $$\left( [0,\pi]\times \mathbb{S}^1\times\mathbb{S}^1 \text{ with coordinates }(t,\phi_1,\phi_2),\ker{\cos{t}\ d\phi_1+sin{t}\ d\phi_2} \right)\rightarrow(M,\xi).$$

In Example~\ref{excontact}~3), for all $n>1$, the contact manifold $(\mathbb{T}^3,\xi_n)$ contains Giroux torsion, while $(\mathbb{T}^3,\xi_1)$ is constructed by gluing two Giroux $\pi$-torsions along their boundary.

\end{remark}

A special case of exact symplectic fillings was observed by Mitsumatsu, in the presence of smooth volume preserving Anosov flows \cite{mitsumatsu} (see Section~\ref{4} for more discussion and improvement of Mitsumatsu's theorem). Alongside \cite{stein}, these were the first examples of exact symplectic fillings with disconnected boundaries. Explicit examples of such structures can be found in \cite{mitsumatsu} and they are also discussed in \cite{mnw}.

\begin{definition}\label{lpair}
We call a pair $(\alpha_-,\alpha_+)$ a {\em Liouville pair}, if $\alpha_-$ and $\alpha_+$ are negative and positive $C^1$ contact forms, respectively, whose kernels are transverse and $[-1,1]_t\times M$, equipped with the symplectic structure $d\{ (1-t)\alpha_-+(1+t)\alpha_+ \})$ is an exact symplectic filling for $(M,\ker{\alpha_+})\sqcup(-M,\ker{\alpha_-})$, where $-M$ is $M$ with reversed orientation.
\end{definition}

We note that in the above definition, the corresponding Liouville form is always at least $C^1$ since the underlying contact structures $\alpha_-$ and $\alpha_+$ are (see Remark~\ref{regrem}).

\section{Anosovity And The Geometry Of Expansion}\label{3}

In this section, we review the basic facts about Anosovity, emphasizing on the expansion behavior of the flows in stable and unstable directions, from a geometric point of view.

In the following, we assume $X$ is a non-zero $C^1$-vector field on a closed, oriented 3-manifold $M$.

\begin{definition}\label{anosov}
We call the $C^1$ flow $\phi^t$ {\em Anosov}, if there exists a splitting $TM=E^{ss} \oplus E^{uu} \oplus \langle X \rangle$, such that the splitting is continuous and invariant under $\phi_*^t$ and 
$$ ||\phi_*^t (v)  || \geq Ae^{Ct}||v || \text{ for any }v \in E^{uu},$$
$$||\phi_*^t (u) || \leq Ae^{-Ct}||u ||\text{ for any }u \in E^{ss},$$
where $C$ and $A$ are positive constants, and $||.||$ is induced from some Riemannian metric on $TM$. We call $E^{uu}$ ($E^{uu}\oplus \langle X\rangle$) and $E^{ss}$ ($E^{ss}\oplus \langle X\rangle$), the strong (weak) unstable and stable directions (bundles), respectively. Moreover, we call the vector field $X$, the generator of such flow, an {\em Anosov vector field}.
\end{definition}

In this paper, we assume $E^{ss}$ and $E^{uu}$ to be orientable. This can be arranged, possibly after going to a double cover of $M$.

\begin{example}
Classic examples of Anosov flows in dimension 3 include the geodesic flows on the unit tangent space of hyperbolic surfaces and suspension of Anosov diffeomorphisms of torus. By now, we know that there are Anosov flows on hyperbolic manifolds as well \cite{hypanosov}.
\end{example}

Here, we note that by \cite{anosov},  a small perturbation of any Anosov flow is Anosov and moreover, is {\em orbit equivalent} to the original flow, i.e. there exists a homeomorphism mapping the orbits of the perturbed flow to the orbits of the original flow. Therefore, for many practical purposes one can assume higher regularity for the flow. However, for our purposes, it suffices for the generating vector field to be $C^1$.

In \cite{mitsumatsu} and \cite{confoliations}, it is shown that $C^1$ Anosov vector fields span the intersection of a pair of transverse positive and negative contact structures, i.e. a {\em bi-contact structure}. However, it is known that the converse is not true. As a matter of fact, non-zero vector fields in the intersection of a bi-contact structure define a considerably larger class of vector fields, namely {\em projectively Anosov vector fields}. By \cite{mitsumatsu,confoliations}, this is equivalent to the following definition:

\begin{definition}\label{ca}
We call a flow $\phi^t$, generated by the $C^1$ vector field $X$, {\em projectively Anosov}, if its induced flow on $TM/\langle X\rangle$ admits a {\em dominated splitting}. That is, there exists a splitting $TM/\langle X \rangle=E^s \oplus E^u$, such that the splitting is continuous and invariant under $\tilde{\phi}_*^t$ and 
$$ ||\tilde{\phi}_*^t (v)  || / ||\tilde{\phi}_*^t (u) || \geq Ae^{Ct}||v || /||u ||$$
for any $v \in E^u$ ({\em unstable direction}) and $u \in E^s$ ({\em stable direction}), where  $C$ and $A$ are positive constants, $||.||$ is induced from some Riemannian metric on $TM / \langle X\rangle$ and $\tilde{\phi}_*^t$ is the flow induced on $TM/\langle X \rangle$.

Moreover, we call the vector field $X$, a {\em projectively Anosov vector field}.
\end{definition}

Similar to Anosov flows, we assume the orientability of the stable and unstable directions of projectively Anosov flows in this paper.

In \cite{mitsumatsu,confoliations}, it is shown:

\begin{proposition}\label{cabi}
Let $X$ be a $C^1$ vector field on $M$. Then, $X$ is projectively Anosov, if and only if, there exist positive and negative contact structures, $\xi_+$ and $\xi_-$ respectively, which are transverse and $X\subset \xi_+ \cap \xi_-$.
\end{proposition}

 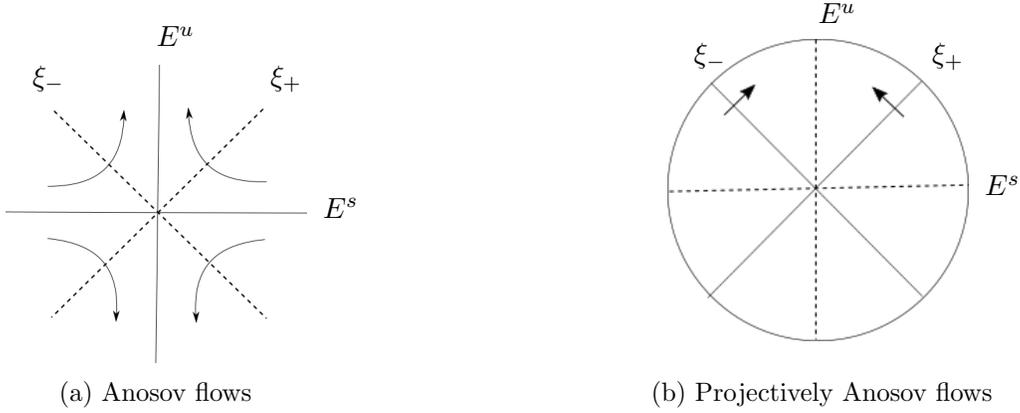
\begin{figure}

 \begin{subfigure}[b]{0.4\textwidth}

  \center \begin{overpic}[width=4cm]{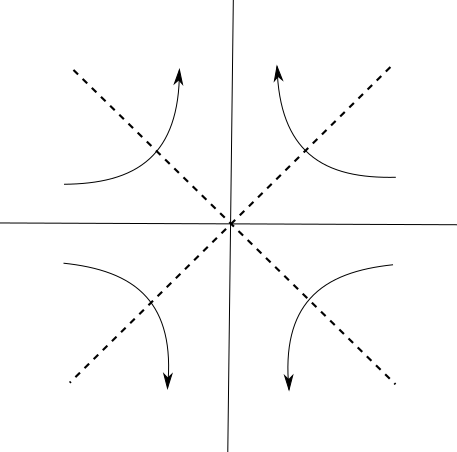}

         \put(10,105){$\xi_-$}
         \put(100,105){$\xi_+$}
         \put(57,120){$E^u$}
         \put(120,55){$E^s$}
       \end{overpic}
    \caption{Anosov flows}
    \label{fig:1}
  \end{subfigure}
  \hspace{2cm}
  \begin{subfigure}[b]{0.4\textwidth}
  \center \begin{overpic}[width=4cm]{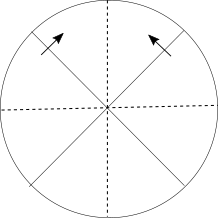}
  \put(57,120){$E^u$}
         \put(120,55){$E^s$}
                  \put(10,105){$\xi_-$}
         \put(100,105){$\xi_+$}

  \end{overpic}
    \vspace{0.3cm}
    \caption{Projectively Anosov flows}
    \label{fig:2}
  \end{subfigure}
  
 \caption{The local behavior of (projectively) Anosov flows}
\end{figure}

This motivates the following definitions:

\begin{definition}
We call the pair $(\xi_-,\xi_+)$ a {\em bi-contact structure} on $M$, if $\xi_+$ and $\xi_-$ are positive and negative contact structures on $M$, respectively, and $\xi_-\pitchfork\xi_+$.
\end{definition}

\begin{definition}
Let $X$ be a projectively Anosov vector field on $M$. We call a bi-contact structure $(\xi_-,\xi_+)$ a {\em supporting} bi-contact structure for $X$, or the generated projectively Anosov flow, if $X\subset \xi_- \cap \xi_+$. We call a positive (negative) contact structure or more generally, any plane field $\xi$, a supporting positive (negative) contact structure or plane field, respectively, for $X$ or the generated flow, if $X \subset \xi$.
\end{definition}

See Example~\ref{exbicontact} for explicit examples of projectively Anosov flows (which are not Anosov). Similar examples can be constructed on Nil manifolds as well \cite{mitsumatsu}. Also, see \cite{bbp} for using the idea of {\em hyperbolic plugs} \cite{building} to construct of projectively Anosov flows (which cannot be deformed, through projectively Anosov flows, into Anosov flows) on atoridal manifolds.

Consider the vector bundle $\pi:TM \rightarrow TM / \langle X \rangle$ and notice that for any plane field $\eta$ which is transverse to the flow, there exists a natural vector bundle isomorphism $TM / \langle X \rangle \simeq \eta$, induced by projection onto $\eta$ and along $X$. Therefore, $\pi$ can be interpreted as such projection as well.

We also notice that in Definition~\ref{ca}, the line fields $E^u,E^s \subset TM/\langle X \rangle$ do not necessarily lift to invariant line fields $E^{uu},E^{ss}\subset TM$, respectively (see \cite{noda} for the examples of when they do not). However, it is a classical fact from dynamical systems that when the induced flow on $TM/\langle X \rangle$ (usually called {\em Linear Poincaré Flow}) admits an invariant continuous {\em hyperbolic splitting} $E^s\oplus E^u$  (uniformly contracting along $E^s$ and expanding along $E^u$), such lift does exist and we we will have an invariant splitting as in Definition~\ref{anosov} (see \cite{hyplift}, Proposition~1.1).

\begin{definition}\label{balanced}
We call a projectively Anosov flow (vector field) {\em balanced}, if it preserves a transverse plane field $\eta$.
\end{definition}

\begin{proposition}\label{balanced2}
The flow $\phi^t$ is a balanced projectively Anosov, if and only if, there exists a splitting $TM=E^{ss} \oplus E^{uu} \oplus \langle X \rangle$, such that the splitting is continuous and invariant under $\phi_*^t$ and 
$$ ||\phi_*^t (v)  || / ||\phi_*^t (u) || \geq Ae^{Ct}||v || /||u ||$$
for any $v \in E^{uu}$ ({\em strong unstable direction}) and $u \in E^{ss}$ ({\em strong stable direction}), where  $C$ and $A$ are positive constants, and $X$ is the $C^1$ generator of the flow.
\end{proposition}

\begin{proof}
We can easily observe $E^{ss}=\eta\cap \pi^{-1}(E^s)$ and $E^{uu}=\eta\cap \pi^{-1}(E^u)$.
\end{proof}

\begin{remark}
 It can be seen that given a projectively Anosov flows, the plane fields $\pi^{-1}(E^s)$ and $\pi^{-1}(E^u)$ are $C^0$ (possibly non-uniquely) integrable plane fields, tangent to (branching) foliations we call {\em stable and unstable foliations}, respectively. In the Anosov case, thanks to the classical regularity theory of Anosov flows, higher regularity of these foliations can be assumed and that is the basis for the use of foliation theory to study Anosov dynamics. Therefore, such tools are not all well transferred to projectively Anosov dynamics in general. However, assuming more regularity for the associated foliations of a projectively Anosov flow, some rigidity results are known \cite{noda,asaoka}.
\end{remark}

It is worth to pause and make few observations about the geometry of projectively Anosov flows (the remark is discussed more in depth in \cite{confoliations}).

\begin{remark}\label{geomca}
If $X$ is some vector field on $M$, which is tangent to some plane field $\xi$, we can measure the contactness of $\xi$ from the rotation of the flow, with respect to $\xi$, in the following way. Choose some transverse plane field $\eta$, which is differentiable in the direction of $X$ (for instance, if $X$ is a balanced projectively Anosov vector field, $E^{ss}\oplus E^{uu}$ can be chosen) and orient it such that $X$ and $\eta$ induce the chosen orientation of $M$. Let $\lambda=\xi \cap \eta$ and $\lambda_p^t=\phi_*^{-t}(\xi_{\phi^t(p)})\cap \eta$ for $x\in M$ and $t\in \mathbb{R}$. Finally, let $\theta_p^t$ be the angle between $\lambda^0_p$ and $\lambda_p(t)$, for some Riemannian metric, which is differentiable in the direction of $X$. Then, $\xi$ is a positive or negative contact structure, if and only if,

$$X\cdot \theta_p(t)<0 \hspace{0.2cm} \text{ or }\hspace{0.2cm}  X\cdot \theta_p(t)>0,$$
respectively, for all $p$ and $t$.

Now, if $X$ is a projectively Anosov vector field, and $(\xi_-,\xi_+)$ a bi-contact structure such that $X\subset \xi_-\cap \xi_+$, let $\eta$ be any transverse plane field, and  $\lambda_+=\xi_+\cap \eta$ and $\lambda_-=\xi_- \cap \eta$. Similar to above, we can define $\lambda_{+,p}^t$ and $\lambda_{-,p}^t$ and observe
$$\lim_{t \rightarrow +\infty}{\lambda_{+,p}^t}=\lim_{t \rightarrow +\infty}{\lambda_{-,p}^t}=\pi^{-1}(E^s)\cap \eta$$
and
$$\lim_{t \rightarrow -\infty}{\lambda_{-,p}^t}=\lim_{t \rightarrow -\infty}{\lambda_{+,p}^t}=\pi^{-1}(E^u)\cap \eta,$$

Equivalently,

$$\lim_{t\rightarrow +\infty} \phi^t_*(\xi_+)=\lim_{t\rightarrow +\infty} \phi^t_*(\xi_-)=\pi^{-1}(E^u)$$
and
$$\lim_{t\rightarrow -\infty} \phi^t_*(\xi_+)=\lim_{t\rightarrow -\infty} \phi^t_*(\xi_-)=\pi^{-1}(E^s).$$
\end{remark}

It turns out that we can characterize Anosovity of a projectively Anosov vector field by the {\em expansion of its stable and unstable directions}.

We first note that the norm used in the definition of a (projectively) Anosov flow $X$ is in general induced from some $C^0$ Riemannian structure $g$. However, if we replace $g$ with $\frac{1}{T}\int_0^T \phi^{t*}gdt$, where $\phi^t$ is the flow of $X$, the resulting Riemannian metric will be differentiable in $X$-direction, i.e. $\mathcal{L}_Xg^T$ would exist. Moreover, by considering large enough $T$, with respect to such metric, we can assume $A=1$ in the above definitions, meaning that the expansion or contraction in unstable and stable directions, respectively, for an Anosov flow, or the relative expansion for a projectively Anosov flow, start immediately. Assuming such conditions, we can compute the infinitesimal rate of expansion for vectors in the stable and unstable directions. Remember that any transverse plane field $\eta$ induces a vector bundle isomorphism $TM/\langle X \rangle$. Using such isomorphism, the restriction $g |_\eta$ of any Riemannian metric $g$ on $TM$, defines a metric on $TM/\langle X \rangle$, and conversely, given any Riemannian metric on $TM/\langle X \rangle$, we can define a metric on $TM$, whose restriction on $\eta$ is induced from such metric.

Let $\tilde{e}_u\in E^u \subset TM/\langle X \rangle$ be the unit vector field (with respect to some Riemannian metric) defined in the neighborhood of a point. Noticing that the linear flow on $TM/\langle X \rangle$ preserves the direction of $\tilde{e}_u$, we compute:
$$\mathcal{L}_X \tilde{e}_u=\frac{\partial}{\partial t} \tilde{\phi}_*^{-t} (\tilde{e}_u)\bigg|_{t=0}=\frac{\partial}{\partial t} \frac{\tilde{\phi}_*^{-t} \left(\tilde{\phi}_*^t (\tilde{e}_u) \right)}{||\tilde{\phi}_*^t (\tilde{e}_u)||}\bigg|_{t=0}=\left(\frac{\partial}{\partial t} \frac{1}{||\tilde{\phi}_*^t (\tilde{e}_u)||} \right)\bigg|_{t=0}\tilde{e}_u$$
$$=-\left(\frac{\partial}{\partial t} ||\tilde{\phi}_*^t (\tilde{e}_u)|| \right)\bigg|_{t=0}\tilde{e}_u=-\left(\frac{\partial}{\partial t} \ln{||\tilde{\phi}_*^t (\tilde{e}_u)||} \right)\bigg|_{t=0}\tilde{e}_u.$$

We can do similar computation for the (locally defined) unit vector field $\tilde{e}_s\in E^s$.

\begin{definition}
Using the above notation, we define {\em the expansion rate of the (un)stable direction} as
$$r_s:=\frac{\partial}{\partial t} \ln{||\tilde{\phi}_*^t (\tilde{e}_s)||}\bigg|_{t=0} \text{   } \left ( r_u:=\frac{\partial}{\partial t} \ln{||\tilde{\phi}_*^t (\tilde{e}_u)||}\bigg|_{t=0} \right).$$
\end{definition}

We note that similar notions have been previously used and studied in the literature. For instance, see \cite{regular,simic}. Naturally, the positive and negative expansion rate correspond to the expanding and contracting behaviors of the flow in a certain direction, respectively.

\begin{proposition}\label{cageoformula}
The above computation shows:
$$\mathcal{L}_X \tilde{e}_s=-r_s \tilde{e}_s \text{   }\left( \mathcal{L}_X \tilde{e}_u=-r_u \tilde{e}_u  \right),$$
and
$$\tilde{\phi}^T_*(\tilde{e}_s)=e^{\int_0^T r_s(t)dt}\tilde{e}_s\text{   }\left( \tilde{\phi}^T_*(\tilde{e}_u)=e^{\int_0^T r_u(t)dt}\tilde{e}_u  \right).$$
\end{proposition}

Now consider a transverse plane field $\eta\simeq TM/\langle X \rangle$, equipped with a Riemannian metric $\tilde{g}$ on $TM/\langle X \rangle$, defining stable and unstable expansion rate of $r_s,r_u$. From $\tilde{g}$, a Riemannian metric is induced on $\eta$, which can be extended to a Riemannian metric $g$ on $TM$, such that $\mathcal{L}_X g$ exists, assuming that $\mathcal{L}_X \hat{g}$ and $\mathcal{L}_X \eta$ exist. Let $e_s,e_u\in\eta$ be chosen such that $\pi(e_s)=\tilde{e}_s$ and $\pi(e_u)=\tilde{e}_u$, and notice that $||e_s||=||e_u||=1$. Let $\pi_\eta$ be the projection onto $\eta$ and along $X$ and compute
$$\mathcal{L}_Xe_u=\frac{\partial}{\partial t} \phi_*^{-t} (e_u)\bigg|_{t=0}=\frac{\partial}{\partial t} \pi_\eta\big( \phi_*^{-t} (e_u)\big) \bigg|_{t=0}+q^\eta_u X,$$
for some function $q^\eta_u:M\rightarrow \mathbb{R}$. Since the metric on $\eta$ is induced from $\hat{g}$, this implies
$$\mathcal{L}_Xe_u=-r_ue_u+q^\eta_uX.$$

Similarly, $$\mathcal{L}_Xe_s=-r_se_s+q^\eta_sX,$$
for some function $q^\eta_s:M\rightarrow \mathbb{R}$. We have proved:

\begin{proposition}\label{cageoformula2}
Let $X$ be a projectively Anosov vector field with $r_s,r_u$ being its expansion rate of stable and unstable directions (with respect to some metric on $TM/\langle X \rangle$), respectively. Then, for any transverse plane field $\eta$, there exists a metric on $TM$ such that for unit vector fields $e_u\in\eta \cap \pi^{-1}(E^u)$ and $e_s\in\eta \cap \pi^{-1}(E^s)$, we have
$$\mathcal{L}_Xe_u=-r_ue_u+q^\eta_uX,$$
and $$\mathcal{L}_Xe_s=-r_ue_s+q^\eta_sX,$$
for appropriate real functions $q^\eta_u,q^\eta_s:M\rightarrow \mathbb{R}$. 
\end{proposition}

We also observe the following fact, which we will use in the proof of Theorem~\ref{main}:

\begin{proposition}\label{balancedform}
Let $X$ be a projectively Anosov vector field. When $X$ is balanced (in particular, when $X$ is Anosov), there exists a transverse plane field $\eta$ as in Proposition~\ref{cageoformula2}, for which $q_u^\eta=q_s^\eta=0$ everywhere. In this case,
$$\mathcal{L}_X e_s=-r_s e_s \text{   }\left( \mathcal{L}_X e_u=-r_u e_u  \right),$$
and
$$\phi^T_*(e_s)=e^{\int_0^T r_s(t)dt}e_s\text{   }\left( \phi^T_*(e_u)=e^{\int_0^T r_u(t)dt}e_u  \right).$$
\end{proposition}

The definition of projectively Anosov vector fields implies:

\begin{proposition}\label{rateca}
Let $X$ be a projectively Anosov vector field and $r_s$ and $r_u$, the expansion rates of stable and unstable directions, respectively, with respect to any Riemannian metric, satisfying the metric condition of Definition~\ref{ca} with $A=1$, which is differentiable in $X$-direction, then 
$$r_u-r_s>0.$$
\end{proposition}

\begin{proof}
Since $X$ is projectively Anosov, the exists a Riemannian metric $g$, such that $\mathcal{L}_X g$ exists and

$$ ||\tilde{\phi}_*^t (\tilde{e}_u)  || / ||\tilde{\phi}_*^t (\tilde{e}_s) || \geq e^{Ct}||\tilde{e}_u || /||\tilde{e}_s ||$$
where $\tilde{\phi}^t$ is the flow of $X$, $||.||$ is the norm on $TM/\langle X \rangle$, induced from $g$, $\tilde{e}_u\in E^u$ and $\tilde{e}_s\in E^s$ are unit vectors, and $C$ is a positive constants. Therefore,
$$\ln{||\tilde{\phi}_*^t (\tilde{e}_u)  ||} -\ln{ ||\tilde{\phi}_*^t (\tilde{e}_s) ||} \geq Ct$$
and
$$r_u-r_s=\frac{\partial}{\partial t}\ln{||\tilde{\phi}_*^t (\tilde{e}_u)  ||} \bigg|_{t=0} -\frac{\partial}{\partial t}\ln{||\tilde{\phi}_*^t (\tilde{e}_s)  ||}\bigg|_{t=0}\geq C>0.$$
\end{proof}

\begin{remark}
In proof of Theorem~\ref{main}, we will also see that converse of the above proposition also holds, in the sense that given a $C^1$ projectively Anosov vector field, for any Riemannian metric with $r_u-r_s>0$, the plane fields $\langle X, \pi^{-1}(\frac{e_u +e_s}{2})\rangle$ and $\langle X, \pi^{-1}(\frac{e_u -e_s}{2})\rangle$ define positive and negative contact structures, respectively, possibly after a perturbation to make the plane fields $C^1$.
\end{remark}

Similar computation, using the definition of Anosov flows and the fact that hyperbolicity of $TM/\langle X \rangle$ implies Anosovity of the flow (\cite{hyplift}, Proposition~1.1), yield:
\begin{proposition}\label{rateanosov}
Let $X$ be a projectively Anosov vector field and $r_s$ and $r_u$. Then $X$ is Anosov, if and only if, with respect to some Riemannian metric, we have $$r_u>0>r_s.$$
\end{proposition}

\begin{remark}

The above computation also shows that both Anosovity and projective Anosovity are preserved under reparametrizations of the flow. More precisely, let $X$ is projectively Anosov vector field with expansion rates of $r_s$ and $r_u$, in the stable and unstable directions, respectively (with respect to some metric). Then, for any positive function $f:M\rightarrow \mathbb{R}^{>0}$, the vector field $fX$ has expansion rates of $fr_s$ and $fr_u$, in the stable and unstable directions, respectively (with respect to the same metric). Therefore, the conditions of both Proposition~\ref{rateca} and Proposition~\ref{rateanosov}, are preserved under such transformations.

\end{remark}


\section{Contact and Symplectic Geometric Characterization of Anosov Flows}\label{4}

The goal of this section is to give a purely contact and symplectic geometric characterization of Anosov flows

\begin{remark}\label{reg}
In \cite{mitsumatsu}, Mitsumatsu shows that the generator vector field of any smooth volume preserving Anosov flow lies in the intersection of a pair of transverse negative and positive contact structures, admitting contact forms $\alpha_-$ and $\alpha_+$, respectively, such that $(\alpha_-,\alpha_+)$ is a Liouville pair (see Definition~\ref{lpair}). Beside the symmetry induced by the existence of an invariant volume form (see \cite{hoz3}), the crucial ingredient is the fact that the weak stable and unstable bundles are known to be at least $C^1$ \cite{hruder}. Therefore the Anosovity of the flow can be translated easily to the differential geometry, and in particular, the contact geometry of the underlying manifold (note that contact structures are at least $C^1$). We remark that although it is known now that these invariant bundles are $C^1$ for any smooth Anosov flow in dimension 3 \cite{regular}, such plane fields are only Hölder continuous, if we want to generalize the result to Anosov flows of lower regularity.
In the following, we improve Mitsumatsu's results by showing that the same holds without any regularity assumption on the weak bundles, using careful approximations of these plane fields. We note that to prove the converse of this statement, these approximations are necessary even for smooth flows, since Anosovity of the flow is not assumed (and the weak bundles of smooth projectively Anosov flows are not necessarily $C^1$). Moreover, we believe that the applications of these approximation techniques can be furthered to other questions about Anosov flows, even when the invariant bundles are $C^1$, since in that case, they facilitate controlling the second variations of these plane fields (and therefore, the associated Reeb vector fields for the underlying contact structures) along the flow. See \cite{hoz3} for such application in the surgery theory of Anosov flows.
\end{remark}

\begin{theorem}\label{main}

Let $\phi^t$ be a flow on the 3-manifold $M$, generated by the $C^1$ vector field $X$. Then $\phi^t$ is Anosov, if and only if, $\langle X \rangle= \xi_+ \cap \xi_-$, where $\xi_+$ and $\xi_-$ are transverse positive and negative contact structures, respectively, and  there exist contact forms $\alpha_+$ and $\alpha_-$ for $\xi_+$ and $\xi_-$, respectively, such that $(\alpha_-,\alpha_+)$ and $(-\alpha_-,\alpha_+)$ are Liouville pairs.

\end{theorem}

\begin{proof}

We begin by assuming $\phi^t$ is Anosov.

Let $g$ be the $C^0$ Riemannian metric for which the condition of Anosovity is satisfied and $g(X,E^{ss})=g(X,E^{uu})=g(E^{ss},E^{uu})=0$. After replacing $g$ with $\frac{1}{T}\int_0^T \phi^{t*}g(t)dt$ for large $T$, we can assume the same orthogonality conditions hold, $\mathcal{L}_X g$ exists everywhere and the expansion and contraction of $E^{u}u$ and $E^{ss}$ start immediately (i.e. can assume $A=1$ in the Definition~\ref{anosov}). This means that if $e_s\in E^{ss}$ and $e_u\in E^{uu}$ are unit vector fields, the stable and unstable expansion rates, $r_s$ and $r_u$ are defined and are negative and positive, respectively (Proposition~\ref{rateanosov}).
Moreover, choose such $e_s$ and $e_u$ so that $(e_s,e_u,X)$ is an oriented basis for $M$ as in Figure~(2).

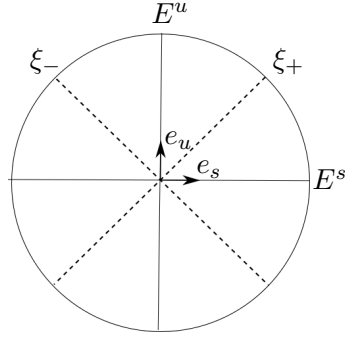
\begin{figure}\label{fig1}
  \center \begin{overpic}[width=4cm]{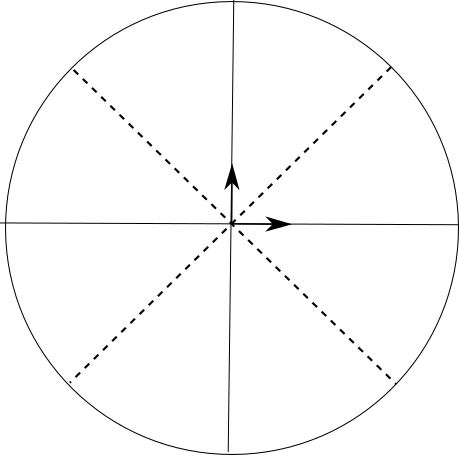}
  \put(8,100){$\xi_-$}
  \put(100,100){$\xi_+$}
    \put(54,117){$E^u$}
      \put(115,54){$E^s$}
       \put(71,60){$e_s$}
        \put(59,71){$e_u$}

  \end{overpic}
\caption{$TM/\langle X \rangle \simeq E^s\oplus E^u$}
\end{figure}

Let $\alpha^{pre}_u$ and $\alpha^{pre}_s$ be approximations of ${\hat e}_u$ and ${\hat e}_s$, the $g$-duals of $e_u$ and $e_s$, respectively, by $C^1$ 1-forms, such that $\alpha^{pre}_u(X)=\alpha^{pre}_s(X)=0$. Note that we are not expecting {\em good} approximations at this point -these are only $C^0$-approximations-, as we only require to have $\alpha^{pre}_u(e_u)>0$ and $\alpha^{pre}_s(e_s)>0$. To have this, we only need to appropriately approximate $E^s$ or $E^u$ as a line bundle in $TM/\langle X \rangle$. The sum of such line bundles with $\langle X \rangle$ yields the desired $C^1$ plane fields. Eventually, we will obtain our {\em good} approximations by flowing $\alpha^{pre}_u$ and $\alpha^{pre}_s$ along the flow, as we will later see.


There exist continuous functions $f_u$ and $f_s$ which are differentiable in direction of $X$ and $f_u\alpha^{pre}_u (e_u)=f_s\alpha^{pre}_s(e_s)=1$. We can $C^0$-approximate $f_u$ and $f_s$ with smooth functions $\tilde{f}_u$ and $\tilde{f}_s$, such that $|X \cdot f_u -X\cdot \tilde{f}_u|$ and $|X\cdot f_s -X\cdot \tilde{f}_s|$ are arbitrary small. This can be achieved using a delicate partition of unity, which respects the differentiation in the direction of the flow (as we will shortly see, in fact we will need a parametric version of such approximation).

\begin{lemma}\label{app}
Let $\phi^t$ be any non-singular $C^{1}$ flow on a closed manifold $M$ (of arbitrary dimension) and $f$ be any $X$-differentiable continuous function. For any $\epsilon>0$, there exists a $C^1$ function $\tilde{f}$, such that we have
$$|f-\tilde{f}|<\epsilon\ \ \ \ \ \text{and}\ \ \ \ \ |X\cdot f-X\cdot \tilde{f}|<\epsilon$$
everywhere.
\end{lemma}

\begin{proof}
We can find such function using local solutions and a partition of unity. However, since we want to control the derivative of such function in the direction of the flow, we need to control the parameters of our partition of unity carefully.

Fix $\epsilon>0$. Consider the collection $\{ (U_i,V_i,\Sigma_i,\tau_i)\}_{1\leq i \leq N}$, where $\Sigma_i$s are $C^1$ local sections of the flow $\phi^t$ (which are open disks) and for some $\epsilon_i>0$, we can define the open flowboxes $U_i:=\{\phi^t(x)\ \text{s.t.}\ x\in \Sigma_i, -\tau_i<t<\tau_i \}$ and $V_i:=\{\phi^t(x)\ \text{s.t.}\ x\in \Sigma_i, -\frac{\tau_i}{2}<t<\frac{\tau_i }{2}\}\subset U_i$, such that $\{ V_i\}_{1\leq i\leq N}$ is a covering for $M$. Notice, that we can find such covering, since for any $x\in M$, we can find such $(U_x,V_x,\Sigma_x,\tau_x)$, where $x\in \Sigma_x$ and $\epsilon_x$ is sufficiently small. Compactness of $M$ implies that finitely many of $V_x$ s cover $M$. Therefore, we get the desired collection.

Let $\{ \psi_i \}_{1\leq i\leq N}$ be a partition of unity with respect to such covering. In particular, we have $supp(\psi_i)\subset V_i$. Note that there exists a compact subset $\tilde{\Sigma}_i\subset \Sigma_i$, such that $supp(\psi_i)\subset \{\phi^t(x)\ \text{s.t.}\ x\in \tilde{\Sigma}_i, -\frac{\tau_i}{2}<t<\frac{\tau_i }{2}\}$. Define $h_i:\mathbb{R}\rightarrow \mathbb{R}$ to be a bump function such that $h_i([\frac{\tau_i}{2},\frac{\tau_i}{2}])=1$, $h_i((-\infty,{\tau_i}]\cup[\tau_i,+\infty))=0$ and $h_i$ is monotone elsewhere. Finally, let $C_i:=\sup{|\frac{dh_i}{dt}|}$ and choose some positive real number $\delta_i\in \mathbb{R}$, such that $\max{\{\delta_i,C_i\delta_i\}}<\frac{\epsilon}{N}$.

We can write $f=\sum_{1\leq i\leq N}\psi_i f$. Let $g_i$ be a $C^1$ function defined on $\Sigma_i$, such that $|g_i-\psi_i f|<\delta_i\big|_{\Sigma_i}$ and $g_i=0$ on $\Sigma_i/ \tilde{\Sigma}_i$. Now we can extend $g_i$ to $U_i$ by solving the differential equation $X\cdot g_i=X\cdot (\psi_i f)$ on $\Sigma_i$. Note that we have $|g_i-\psi_i f|<\delta_i$, everywhere on $U_i$. In particular, $|g_i|<\delta_i$ on $U_i \backslash V_i$.

We can then define $\tilde{g}_i$ on $M$, by letting $\tilde{g}_i(\phi^t(x))=h_i(t)g_i(\phi^t(x))$ for any $x\in\Sigma_i$ and $\tilde{g}_i=0$ on $M\backslash U_i$. Note that $|\tilde{g}_i -\psi_if|<\delta_i$. Moreover, since $X\cdot\tilde{g}_i=(X\cdot h_i)g_i+h_i(X\cdot g_i)$, we have $|X\cdot\tilde{g}_i-X\cdot(\psi_i f)|=|X\cdot g_i-X\cdot(\psi_i f)|=0$ on $\tilde{V}_i$, $|X\cdot\tilde{g}_i-X\cdot(\psi_i f)|=|(X\cdot h_i) g_i|<C_i\delta_i$ on $U_i\backslash \tilde{V}_i$ and $X\cdot \tilde{g}_i=X\cdot (\psi_i f)=0$ elsewhere. Therefore, we have $|X\cdot\tilde{g}_i-X\cdot(\psi_i f)|<C_i\delta_i$ everywhere. Now, we can see that $\tilde{f}:=\sum_{1 \leq i \leq N}\tilde{g}_i$ is the desired function, since it is $C^1$ by construction and we have
$$|\tilde{f}-f|<\sum_{1\leq i \leq N} |\tilde{g}_i-\psi_if|<\sum_{1 \leq i \leq N}\delta_i<\epsilon$$
and
$$|X\cdot\tilde{f}-X\cdot f|<\sum_{1\leq i \leq N} |X\cdot\tilde{g}_i-X\cdot(\psi_if)|<\sum_{1 \leq i \leq N}C_i\delta_i<\epsilon.$$

\end{proof}

In the following, when there is no confusion, for any point $x\in M$, we refer to $r_s(x)$ by $r_s$ or $r_s(0)$ and to $r_s(\phi^t(x))$ by $r_s(t)$. Similarly, for other functions in this proof.

Using the above lemma and fixing $\epsilon>0$, for any $T\geq 0$ we can find functions $\tilde{f}^T_u$ and $\tilde{f}^T_s$ which are differentiable along $X$, satisfy
$$\begin{cases}
|\tilde{f}^T_u-f_u|<\frac{\epsilon}{T} \ \ \text{and} \ \
|X\cdot \tilde{f}^T_u-X\cdot f_u|<\frac{\epsilon}{T} \\
|\tilde{f}^T_s-f_s|<\frac{\epsilon}{T} \ \ \text{and} \ \
|X\cdot \tilde{f}^T_s-X\cdot f_s|<\frac{\epsilon}{T}
\end{cases}
$$
everywhere
and furthermore, the 1-forms defined as
$$\begin{cases}
\alpha^T_u:=I^T_u\phi^{T*}(\tilde{f}^T_u\alpha_u^{pre}) \ \ \ \text{where} \ \ \ I^T_u:=e^{-\int_0^T r_u (t)dt} \\
\alpha^T_s:=I^{-T}_s\phi^{-T*}(\tilde{f}^{T}_s\alpha_s^{pre}) \ \ \ \text{where} \ \ \ I^T_s:=e^{-\int_0^T r_s (t)dt} \\
\end{cases},$$
are $C^1$. More specifically, we can naturally identify the conformal class of the 1-form $\phi^{T*}(\alpha_u^{pre})$ with the space of functions on $M$ and then apply Lemma~\ref{app} to get an approximation -in the above sense- of $I^T_u\phi^{T*}(f_u\alpha_u^{pre})$ by a $C^1$ 1-form which we denote by $\alpha_u^T$. Then, there exists a function $\tilde{f}_u^T$ using which we can write $\alpha_u^T=I^T_u\phi^{T*}(\tilde{f}^T_u\alpha_u^{pre})$, since $\ker{\alpha_u^T}=\ker{\phi^{T*}(\alpha_u^{pre})}$ (note that we fixed the conformal class of a 1-form with $C^1$ kernel while approximating). Similarly, we can construct $\alpha_s^T$. It is helpful to think of the approximating 1-forms $\alpha_u^T$ and $\alpha_s^T$ in terms of their decomposition into $I_u^T$ and $I_s^T$, which contains the expansion information, and the functions $\tilde{f}_u^T$ and $\tilde{f}_s^T$, whose role is to adjust the regularity of the resulting $\alpha_u^T,\alpha_s^T$ to be at least $C^1$. All these terms are a priori not $C^1$, while differentiable along the flow, since $I_u^T$ is in general only $C^0$ and differentiable along the flow. Note that $\tilde{f}_u^T$ is exactly as regular as $I^T_u$ and in particular, its regularity class depends on $T$. It turns out, as we will see, that the only derivatives of these 1-forms appearing in our computations are along the flow. Therefore, it is convenient to use such formulation which allows to keep track of the expansion information.

This implies that the quantities $|\tilde{f}_u^T\alpha_u^{pre}(e_u)-1|$, $|X\cdot (\tilde{f}_u^T\alpha_u^{pre}(e_u)|$ and the like for the stable bundle can be taken to be arbitrarily uniformly small. We record this fact by stating that for any fixed $T$, we have the following uniform convergences
\begin{equation}\label{cond1}
\lim_{T\rightarrow \infty} [\tilde{f}_u^T\alpha_u^{pre}(e_u)-1]=
\lim_{ T\rightarrow \infty} [X\cdot (\tilde{f}_u^T\alpha_u^{pre}(e_u))]=0
\end{equation}
and
\begin{equation}\label{cond2}
\lim_{T\rightarrow \infty} [\tilde{f}_s^T\alpha_s^{pre}(e_s)-1]=
\lim_{T\rightarrow \infty} [X\cdot (\tilde{f}_s^T\alpha_s^{pre}(e_s))]=0,
\end{equation}
by which we mean that there exists a family of functions $\tilde{f}_u^T$ and $\tilde{f}_s^T$ for $T>0$ approximating $f_u$ and $f_s$ in the above sense, where the uniform convergence, when $T\rightarrow \infty$, to $f_u$ and $f_s$ can be described as \ref{cond1} and \ref{cond2}.





\begin{claim}\label{claim1}
We have
$$\begin{cases}
\alpha^T_u(e_u(0))=\tilde{f}_u^T\alpha^{pre}_u(e_u(T)) \\
\alpha^T_s(e_s(0))=\tilde{f}_s^T\alpha^{pre}_s(e_s(T)).
\end{cases}
$$
\end{claim}
\begin{proof}
$$\alpha^T_u(e_u(0))=I^T_u \tilde{f}_u^T\alpha^{pre}_u(\phi^T_*(e_u(0)))=I^T_u \tilde{f}_u^T\alpha^{pre}_u(\frac{1}{I^T_u}e_u(T))=\tilde{f}_u^T\alpha^{pre}_u(e_u(T)),$$
where the middle equality is implied by Proposition~\ref{balancedform}. Other implication follows similarly.
\end{proof}

Consequently, by Equations~\ref{cond1} and \ref{cond2}, we have the uniform convergences
$$\begin{cases}
\lim_{T\rightarrow \infty} [\alpha_u^T(e_u)-1]=
\lim_{ T\rightarrow \infty}[ X\cdot (\alpha_u^{T}(e_u)]=0 \\
\lim_{T\rightarrow \infty} [\alpha_s^{T}(e_s)-1]=
\lim_{T\rightarrow \infty} [X\cdot (\alpha_s^{T}(e_s)]=0
\end{cases}.$$
%


%

We then prove the following elementary facts about the functions $I_u^T$ and $I_s^T$.

\begin{claim}\label{claim2}
$$\lim_{T\rightarrow +\infty}\frac{I^T_u}{I^T_s}=\lim_{T\rightarrow +\infty}\frac{I^{-T}_s}{I^{-T}_u}=0.$$
\end{claim}

\begin{proof}
$$\lim_{T\rightarrow +\infty}\frac{I^T_u}{I^T_s}=\lim_{T\rightarrow +\infty} e^{\int_0^T r_s (t)-r_u(t)dt}=0.$$
The last equality follows from projective Anosovity of $X$ (Proposition~\ref{rateca}), implying $r_u-r_s>0$.

Similarly, $$\lim_{T\rightarrow +\infty}\frac{I^{-T}_s}{I^{-T}_u}=\lim_{T\rightarrow +\infty} e^{\int_{-T}^0 r_s (t)-r_u(t)dt}=0.$$.
\end{proof}

\begin{claim}\label{claim3}
$$X\cdot I_u^T =[ r_u(0)-r_u(T)   ] I_u^T;$$
$$X\cdot I_s^T = [ r_s(0)-r_s(T)  ] I_s^T.$$
\end{claim}

\begin{proof}
$$X\cdot I_u^T =\frac{\partial}{\partial h} e^{-\int_0^T r_u (t+h)dt} \bigg|_{h=0}$$
$$=[-\int_0^T r'_u (t)dt]I^T_u=[ r_u(0)-r_u(T)   ] I_u^T.$$
The other implication follows similarly.

\end{proof}

Now, using the above calculations, we can show that $\ker{\alpha^T_u}$ and $\ker{\alpha^T_s}$, $C^0$-converge to $\pi^{-1}(E^s)=E^{ss} \oplus \langle X \rangle$ and $\pi^{-1}(E^u)=E^{uu} \oplus \langle X \rangle$, respecting certain $C^1$-quantities.

\begin{lemma}\label{l1}
We have $$\lim_{T\rightarrow +\infty} \ker{\alpha^T_u}=\pi^{-1}(E^s)$$ and $$\lim_{T\rightarrow +\infty} \ker{\alpha^T_s}=\pi^{-1}(E^u).$$
\end{lemma}

\begin{proof}
First compute
$$\lim_{T\rightarrow +\infty} \alpha^T_u(e_s(0))=\lim_{T\rightarrow +\infty} I^T_u \tilde{f}_u^T\alpha^{pre}_u (\phi^T_* e_s(0))=\lim_{T\rightarrow +\infty} \frac{I^T_u}{I^T_s} \tilde{f}_u^T\alpha^{pre}_u (e_s(T))=0.$$

The last equality follows from Claim~\ref{claim2} and the fact that $\tilde{f}_u^T\alpha^{pre}_u (e_s)$ is bounded. Similarly, 
$$\lim_{T\rightarrow +\infty} \alpha^T_s(e_u(0))=0.$$
Claim~\ref{claim1} and the fact that $\alpha^T_u(X)=\alpha^T_s(X)=0$ finish the proof.
\end{proof}

Now, we see that certain $C^1$-variations behave nicely under such limiting procedure.

\begin{lemma}\label{l2}
We have the uniform convergence
$$\lim_{T\rightarrow +\infty} \alpha^T_u \wedge d\alpha^T_u=\lim_{T\rightarrow +\infty} \alpha^T_s \wedge d\alpha^T_s=0.$$
\end{lemma}

\begin{proof}
Using Claim~\ref{claim1}, Claim~\ref{claim2} and Claim~\ref{claim3}, compute
$$(\alpha^T_u \wedge d\alpha^T_u)(e_s,e_u,X)=\alpha^T_u(e_s) \left[-X\cdot (\alpha_u^T(e_u))+\alpha_u^T(\mathcal{L}_Xe_u)\right] - \alpha_u^T(e_u)\left[X\cdot (\alpha_u^T(e_s))-\alpha_u^T(\mathcal{L}_X e_s )\right]$$
$$=\tilde{f}_u^T\alpha^{pre}_u(e_s(T))\left[ -X\cdot(\tilde{f}_u^T\alpha^{pre}_u(e_u(T)))-r_u(0) \tilde{f}_u^T\alpha^{pre}_u (e_u(T)) \right] \frac{I^T_u}{I^T_s} $$
$$-\tilde{f}_u^T\alpha^{pre}_u(e_u(T))\bigg[\left(r_u(0)-r_u(T)-r_s(0)+r_s(T)\right)\tilde{f}_u^T\alpha^{pre}_u(e_s(T))$$
$$+X\cdot (\tilde{f}_u^T\alpha^{pre}_u(e_s(T)))+r_s(0) \tilde{f}_u^T\alpha^{pre}_u(e_s(T))\bigg]\frac{I^T_u}{I^T_s}=A(T,x)\frac{I^T_u}{I^T_s},$$
where $A(T,x)$ is a bounded function of $T$ and $x$, thanks to Lemma~\ref{l2} and the uniform convergence of \ref{cond1} and \ref{cond2}.

Claim~\ref{claim2} concludes the implication and similar computation for $\alpha^T_s \wedge d\alpha^T_s$ finishes the proof.
\end{proof}

\begin{remark}
Using Proposition~\ref{cageoformula2}, one can easily check that Claim \ref{claim1}, \ref{claim2}, \ref{claim3} and Lemma~\ref{l1}, \ref{l2} also hold for similar approximations, when the flow is merely projectively Anosov.
\end{remark}

\begin{lemma}\label{l3}
We have the uniform convergence
$$\begin{cases}
 \lim_{T\rightarrow \infty}\alpha^T_u \wedge d\alpha^T_s=r_s\ \hat{e}_s\wedge \hat{e}_u \wedge \hat{X} \\
 \lim_{T\rightarrow \infty} \alpha^T_s \wedge d\alpha^T_u=-r_u\ \hat{e}_s\wedge \hat{e}_u \wedge \hat{X}
\end{cases}.$$.
\end{lemma}

\begin{proof}
$$(\alpha^T_u \wedge d\alpha^T_s)(e_s,e_u,X)=\alpha_u^T (e_s)[-X\cdot (\alpha^T_s(e_u))-\alpha^T_s(-\mathcal{L}_X e_u)]+\alpha^T_u (e_u)[X\cdot (\alpha^T_s(e_s))-\alpha^T_s(\mathcal{L}_X e_s)]$$
$$=\frac{I^T_u}{I^T_s}\frac{I^{-T}_s}{I^{-T}_u} B(T,x)+ \tilde{f}_u^T\alpha^{pre}_u (e_u(T))[X\cdot ( \tilde{f}_s^T\alpha^{pre}_s(e_s(T)))+r_s(0)  \tilde{f}_s^T\alpha^{pre}_s(e_s(T))],$$
where $B(T,x)$ is a bounded function of $T$ and $x$, thanks to the uniform convergence of \ref{cond1} and \ref{cond2}. Using Claim~\ref{claim2}, the first term vanishes in the limit and we will have the uniform convergence
$$\alpha^T_u \wedge d\alpha^T_s\rightarrow r_s\ \hat{e}_s\wedge \hat{e}_u \wedge \hat{X}.$$

Similar computation proves the other limit.
\end{proof}

The upshot of the previous lemmas is that, even though the invariant bundles have low regularity (here only $C^0$ or Hölder continuous), they can be {\em appropriately} approximated by plane fields as regular as the flow (here $C^1$), where {\em appropriate} means controlling variations of the metric along the flow, which is sufficient to control certain geometric quantities. For instance, note that in the computation of $\alpha^T_u\wedge d\alpha_s^T$ in Lemma~\ref{l3}, the only partial derivatives we need to control is the term $\iota_Xd\alpha^T_s=\mathcal{L}_X \alpha^T_s$. Also, observe that $\tilde{f}_s^T\alpha_s^{pre}(e_s(T))$ and $X\cdot(\tilde{f}_s^T\alpha_s^{pre}(e_s(T)))$ can be taken to be uniformly close to $1$ and $0$ respectively, as a consequence of Lemma~\ref{app} (see the discussion before Claim~\ref{claim1}). For the sake of future reference, we codify the approximation result of the previous lemmas in the following:

\begin{lemma}(Approximation of the invariant plane fields)\label{cod}
Choosing any norm on $E^u$ which is differentiable along the flow, inducing the dual 1-form $\hat{e}_u$ with expansion rate $r_u$, i.e. $\mathcal{L}_X\hat{e}_u=r_u\hat{e}_u$, and any $\epsilon>0$, there exists $C^1$ 1-form $\alpha_u$ such that $||\alpha_u-\hat{e}_u||<\epsilon$ and $||\mathcal{L}_X \alpha_u-\mathcal{L}_X\hat{e}_u||=||\mathcal{L}_X \alpha_u-r_u\hat{e}_u||<\epsilon$.
\end{lemma}


Now we have all the ingredients to finish the proof.

Let $\alpha_+^T:=\frac{1}{2} (\alpha^T_u-\alpha^T_s)$ and $\alpha_-^T:=\frac{1}{2} (\alpha^T_u + \alpha^T_s)$. The goal is to show that $(\alpha_-^T,\alpha_+^T)$ and $(-\alpha_-^T,\alpha_+^T)$ are Liouville pairs, for large $T$. By Lemma~\ref{l2} and \ref{l3}, for large $T$:

$$\alpha_+^T \wedge d\alpha_+^T= \frac{1}{4}\left( \alpha_u^T \wedge d\alpha^T_u -\alpha^T_u \wedge d\alpha_s^T -\alpha^T_s \wedge d\alpha_u^T +\alpha^T_s \wedge d\alpha_s^T\right)>0.$$

Therefore, $\alpha_+^T$ is a positive contact form for large $T$. Similar computation shows that $\alpha_-^T$ is a negative contact form for large $T$.

To show that $(\alpha_-^T,\alpha_+^T)$ is a Liouville pair, we need to show that $\omega^T:=d\alpha^T$ is a symplectic form on $M\times [-1,1]$, where $\alpha^T:=\{\alpha_t^T \}_{t\in [-1,1]}$ and $\alpha^T_t:=(1-t)\alpha_-^T +(1+t)\alpha_+^T=\alpha_u^T -t\alpha^T_s$.

Compute for large $T$:
$$\omega^T \wedge \omega^T= (d\alpha_u^T -td\alpha^T_s -dt\wedge \alpha^T_s)\wedge (d\alpha_u^T -td\alpha^T_s -dt\wedge \alpha^T_s)=$$
$$=dt\wedge \{ -2 \alpha^T_s \wedge d\alpha^T_u +2t \alpha_s^T \wedge d\alpha_s^T \}>0.$$

Then, Lemma~\ref{l2} and \ref{l3} 
imply that $\omega^T$ is symplectic for large $T$. More accurately, by Lemma~\ref{cod}, choices can be made such that $\omega^T\wedge \omega^T$ is arbitrarily close to $2r_u dt\wedge \hat{e}_s\wedge \hat{e}_u \wedge \hat{X}$.

To show that $(-\alpha_-^T,\alpha_+^T)$ is a Liouville pair, let $\tilde{\omega}^T:=d\tilde{\alpha}^T$, where $\tilde{\alpha}^T:=\{\tilde{\alpha}_t^T \}_{t\in [-1,1]}$ and $\tilde{\alpha}^T_t:=-(1-t)\alpha_-^T +(1+t)\alpha_+^T=\alpha_s^T -t\alpha^T_u$. Similar computation shows:
$$ \tilde{\omega}^T \wedge \tilde{\omega}^T=dt \wedge \{-2 \alpha^T_u \wedge d\alpha^T_s +2t \alpha_u^T \wedge d\alpha_u^T \} >0,$$
implying that $\tilde{\omega}^T$ is symplectic for large $T$. In fact, by Lemma~\ref{cod}, choices can be made such that $\tilde{\omega}^T\wedge \tilde{\omega}^T$ is arbitrarily close to $-2r_s dt\wedge \hat{e}_s\wedge \hat{e}_u \wedge \hat{X}$. This finishes the proof of one implication.

\vskip1cm

We now consider the other implication.

Note that by Proposition~\ref{cabi}, such flow is projectively Anosov and therefore, we have the splitting $TM/\langle X \rangle\simeq E^s\oplus E^u $. Without loss of generality, assume $\alpha_+$ and $\alpha_-$ induce the same orientation on $\pi^{-1}(E^u)$ and opposite orientations on $\pi^{-1}(E^s)$ (recall that $\pi$ is the fiberwise projection $TM\rightarrow TM/\langle X \rangle$). The goal is to show that for any point in $M$, when constructing the Liouville form by linearly interpolating $\alpha_+$ and $\alpha_-$ (or $-\alpha_-$), the symplectic condition at the {\em time}, when the kernel of the interpolation is $\pi^{-1}(E^s)$ (or $\pi^{-1}(E^u)$), implies $r_u>0$ (or $r_s<0$). Of course, such {\em time} is a continuous function on the manifold. But it turns out that, thanks to the openness of the symplectic condition, suitable approximation by a $C^1$ function suffices.

Orient $\pi^{-1}(E^u)$ such that $\alpha_+(\pi^{-1}(E^u))>0$ and $\alpha_-(\pi^{-1}(E^u))>0$. Also orient $\pi^{-1}(E^s)$ such that $\alpha_+(\pi^{-1}(E^s))<0<\alpha_-(\pi^{-1}(E^s))$.

Putting such regularity concerns aside, the idea is to let $t=\tau_u$ be the time in the interpolation when $\ker{\alpha}=\ker{\{(1-t)\alpha_-+(1+t)\alpha_+\}}$ coincides with $\pi^{-1}(E^s)$. We use $\alpha|_{\{ t=\tau_u\}}$ to define a norm on $E^u$ and show that $d\alpha \wedge d\alpha|_{\{ t=\tau_u\}}=r_u \Omega$, where $r_u$ is the expansion rate with respect to such norm and $\Omega$ is some appropriate positive volume form. Therefore, the symplectic condition for the Liouville pair $(\alpha_-,\alpha_+)$ implies the absolute expansion in the $E^u$-direction. Similar argument about the Liouville pair $(-\alpha_-,\alpha_+)$ implies $r_s<0$, i.e. the absolute contraction in the $E^s$-direction. In practice, dealing with regularity subtleties is more delicate. 

Let $\tau_u(x)$ be the (continuous) function such that
$$\ker{\{ (1-\tau_u)\alpha_- +(1+\tau_u) \alpha_+\}}=\pi^{-1}(E^s),$$
and set
$$\alpha_u:=(1-\tau_u)\alpha_- +(1+\tau_u)\alpha_+.$$

Note that $\tau_u$ and $\alpha_u$ are continuously differentiable along the flow, as $\alpha_-$ and $\alpha_+$ are $C^1$.
Consider a transverse plane field $\eta$ and define $||.|| \big|_{\pi^{-1}(E^u)}$ such that for a unit $e_u$ orienting $\pi^{-1}(E^u)\cap\eta$, we have $\alpha_u(e_u)=1$, noting that $e_u$ is also continuously differentiable along the flow and therefore, the associated expansion rate $r_u$ is defined. Now,
we can rewrite $\alpha_t:=(1-t)\alpha_- +(1+t)\alpha_+$ as
$$\alpha_t =\alpha_u - (t-\tau_u)\beta_s;$$
where $\beta_s =\frac{-\alpha_+ +\alpha_-}{2}$ is a $C^1$ 1-form and $\beta_s(\pi^{-1}(E^s))>0$.

Similarly, considering the Liouville pair $(-\alpha_-,\alpha_+)$, define $||.|| \big|_{\pi^{-1}(E^s)}$ and let $e_s$ be the unit vector orienting $\pi^{-1}(E^s)\cap\eta$. Note that $(e_s,e_u,X)$ is an oriented basis for $M$ (see Figure~2). Using the vector bundle isomorphism $TM/\langle X \rangle\simeq \eta$, we can extend such norm to a Riemannian metric $\hat{g}$ on $TM/\langle X\rangle$ with $\hat{g}(E^s,E^u)=0$.

Let $\alpha^T_u$ and $\alpha^T_s$ be the  $C^0$-approximations of $\alpha_u$ and $\alpha_s$, which are $C^1$, in the same fashion as above
and define $\tau_u^T$, such that
$$\ker{\alpha_u^T}=\ker{\{ (1-\tau^T_u)\alpha_- +(1+\tau^T_u) \alpha_+\}}.$$

Note that $\tau^T_u$ is $C^1$ and we can rewrite
$$\alpha_t=f_u^T \alpha_u^T-(t-\tau_u^T)\beta_s$$
for continuous function $f_u^T=\alpha_t (e_u) \big |_{t=\tau^T_u}$ (but $f_u^T \alpha_u^T$ is $C^1$, since every other term in the above equation is $C^1$).

Observe that $$\lim_{T\rightarrow +\infty} \tau_u^T=\tau_u,$$
since $\lim_{T\rightarrow +\infty}\ker{\alpha^T_u}=\ker{\alpha_u}$ (Lemma~\ref{l1}). Plug in $e_u$ into $\alpha_t$ at $t=\tau^T_u$ to get
$$f_u^T \alpha^T_u(e_u)=1+(\tau_u -\tau_u^T)\beta_s(e_u)$$
and in particular, \begin{equation}\label{eq1}\lim_{T\rightarrow +\infty}f_u^T\alpha^T_u(e_u)=1. \end{equation}

Similarly, plug in $e_s$ into $\alpha_t$ at $t=\tau^T_u$ to get
$$\alpha^T_u (e_s)=\frac{(\tau_u -\tau_u^T)\beta_s(e_s)}{f^T_u}.$$

Compute
$$X\cdot (\alpha_u^T(e_s))=$$
$$\frac{[X\cdot (\tau_u-\tau_u^T)\beta_s(e_s)+(\tau_u-\tau_u^T)X\cdot (\beta_s(e_s))]f_u^T\alpha^T_u(e_u)}{(f_u^T\alpha^T_u(e_u))^2}$$
$$-\frac{[X\cdot (\tau_u-\tau_u^T)\beta_s(e_u)+(\tau_u-\tau_u^T)X\cdot (\beta_s(e_u))](\tau_u -\tau_u^T)\beta_s(e_s)\alpha_u^T(e_u)}{(f_u^T\alpha^T_u(e_u))^2}$$
$$=\frac{A(x)(\tau_u -\tau_u^T)+B(x) X\cdot (\tau_u -\tau_u^T)}{(f_u^T\alpha^T_u(e_u))^2}$$
for bounded functions $A$ and $B=f_u^T \beta_s(e_s)\alpha^T_u(e_u)+\beta_s(e_u)\beta_s(e_s)\alpha^T_u(e_u)(\tau_u -\tau_u^T)$.

Since $\lim_{T\rightarrow +\infty}X\cdot (\alpha_u^T(e_s))=0$ and $B$ is non-zero for large $T$, we have
$$\lim_{T\rightarrow +\infty} X\cdot (\tau_u -\tau_u^T)=0,$$
implying
\begin{equation}\label{eq2}\lim_{T\rightarrow +\infty}X\cdot [f_u^T\alpha^T_u(e_u)]=\lim_{T\rightarrow +\infty} \{ X\cdot (\tau_u -\tau_u^T)\beta_s(e_u) + (\tau_u -\tau_u^T)X\cdot (\beta_s(e_u))\}=0.\end{equation}

Also note that
\begin{equation}\label{eq3}\lim_{T\rightarrow +\infty}X\cdot [f_u^T\alpha^T_u(e_s)]=\lim_{T\rightarrow +\infty}\{ X\cdot (\tau_u -\tau_u^T)\beta_s(e_s) + (\tau_u -\tau_u^T)X\cdot (\beta_s(e_s))\}=0.\end{equation}

Now if $\alpha:=\{\alpha_t\}_{t\in[-1,1]}$ and $\omega:=d\alpha$, compute

$$\omega=d(f_u^T \alpha_u^T) -[dt-d\tau^T_u]\beta_s -(t-\tau^T_u)d\beta_s;$$
$$0<\omega \wedge \omega \big|_{t=\tau_u^T}=dt \wedge 2\{ -\beta_s \wedge d(f^T_u\alpha^T_u)+(t-\tau_u^T)\beta_s\wedge d\beta_s\}\big|_{t=\tau_u^T}=dt \wedge \{ -2\beta_s \wedge d(f_u^T \alpha^T_u)\}.$$

Compute
$$[\beta_s \wedge d(f^T_u\alpha^T_u)](e_s,e_u,X)=$$
$$=\beta_s(e_s)[-X.(f_u^T\alpha^T_u(e_u))-f^T_u\alpha^T_u(-\mathcal{L}_Xe_u)]-\beta_s(e_u)[X.(f^T_u\alpha_u^T(e_s))-f^T_u\alpha^T_u(-\mathcal{L}_X e_s)]$$

Now by \eqref{eq1}, \eqref{eq2}, \eqref{eq3} and Proposition~\ref{cageoformula2}:
$$0<\omega \wedge \omega \big|_{t=\tau_u}=\lim_{T\rightarrow +\infty} \omega \wedge \omega \big|_{t=\tau_u^T}= $$
$$=\lim_{T\rightarrow +\infty}-dt\wedge\beta_s \wedge d(f^T_u\alpha^T_u)=\beta_s(e_s)r_u \ dt\wedge {\hat e}_s \wedge {\hat e}_u \wedge \hat{X}.$$

Therefore, $r_u>0$.

Similarly, it can be shown that $r_s<0$. This shows hyperbolicity of the splitting $TM/\langle X \rangle=E^s\oplus E^u$. By Proposition~1.1 of \cite{hyplift}, this is equivalent to the flow being Anosov.
\end{proof}

\begin{remark}\label{topprop}
By Theorem~\ref{main}, if $(\xi_-,\xi_+)$ is a supporting bi-contact structure for an Anosov flow, then $\xi_-$ and $\xi_+$ are tight (Theorem~\ref{gromov}), strongly symplectically fillable \cite{elfilling,etfilling} and contain no Giroux torsion\cite{gay}. Furthermore, although in general universal tightness is not achieved from symplectic fillability, since any lift of an Anosov flow to any cover, is also Anosov, $\xi_-$ and $\xi_+$ are universally tight in this case.
\end{remark}

\begin{remark}\label{per}
We note that Theorem~\ref{main} provides new geometric tools for understanding the periodic orbits of Anosov flows, in particular regarding the knot theory of such periodic orbits, which there are many unanswered questions about \cite{knotanosov}. More precisely, if $\gamma$ is a periodic orbit of an Anosov flow with supporting bi-contact structure $(\xi_-,\xi_+)$, then $\gamma$ is a {\em Legendrian knot} for both $\xi_-$ and $\xi_+$. Furthermore, $\gamma\times I$ is an {\em exact Lagrangian} in both Liouville pairs, constructed on $M\times I$. These are standard and well-studied objects in contact and symplectic topology and now, those methods can be transferred to the study of such periodic orbits.
\end{remark}

The following theorem shows that the construction of Liouville pairs for Anosov flows in Theorem~\ref{main} can be based on an arbitrary choice of supporting bi-contact structure.

\begin{theorem}
Suppose $\phi^t$ is an Anosov flow on a 3-manifold $M$, generated by a $C^1$ vector field $X$ and let $(\xi_-,\xi_+)$ be any supporting bi-contact structure for $X$. Then, for any choice of orientations for $\xi_-$ and $\xi_+$, there exist contact forms $\alpha_-$ and $\alpha_+$, such that $\ker{\alpha_-}=\xi_-$ and $\ker{\alpha_+}=\xi_+$, respecting the orientations, and $(\alpha_-,\alpha_+)$ is a Liouville pair.
\end{theorem}

\begin{proof}
Without loss of generality, suppose that for the chosen orientations, the kernel of the linear interpolation between the two contact forms passes through $\pi^{-1}(E^s)$ (the other case is similar). Take $\alpha_u$ with $\alpha_u(\pi^{-1}(E^s))=0$ and $|\alpha_u(e_u)|=1$ for unit vector $e_u$ with $r_u>0$ (assuming the norm involved is differentiable along the flow). There exists $\alpha_s$ and $\tilde{\alpha}_s$ such that $\ker{\alpha^{pre}_+}:=\ker{(\alpha_u-\alpha_s})=\xi_+$ and $\ker{\alpha^{pre}_-}:=\ker({\alpha_u+\tilde{\alpha}_s})=\xi_-$ for $C^0$ 1-forms $\alpha^{pre}_+$ and $\alpha^{pre}_-$ (which are differentiable along the flow) and similar to the proof of Theorem~\ref{main}, the contactness of $\xi_+$ and $\xi_-$ implies that the stable expansion rates corresponding to $\alpha_s$ and $\tilde{\alpha}_s$ satisfy $r_u>r_s$ and $r_u>\tilde{r}_s$, respectively. 

As in Theorem~\ref{main} (see Lemma~\ref{cod}), for any $\epsilon>0$, we can find $C^1$ 1-forms $\alpha_+$ and $\alpha_-$, such that $\ker{\alpha_+}=\xi_+$ and $\ker{\alpha_-}=\xi_-$, $||\alpha_+-\alpha_+^{pre}||<\epsilon$, $||\alpha_--\alpha_-^{pre}||<\epsilon$ and $||\mathcal{L}_X\alpha_+-\mathcal{L}_X\alpha_+^{pre}||<\epsilon$, $||\mathcal{L}_X\alpha_--\mathcal{L}_X\alpha_-^{pre}||<\epsilon$. We claim that sufficiently small $\epsilon$ yields the desired Liouville pair $(\alpha_-,\alpha_+)$. Let $\alpha=(1-t)\alpha_-+(1+t)\alpha_+$ and compute
$$d\alpha\wedge d\alpha=2dt\wedge(\alpha_+-\alpha_-)\wedge [(1-t)d\alpha_-+(1+t)d\alpha_+]$$
$$\approx 2dt\wedge (-\alpha_s-\tilde{\alpha}_s)\wedge [(1-t)\hat{X}\wedge(r_u\alpha_u+\tilde{r}_s\tilde{\alpha}_s)+(1+t)\hat{X}\wedge(r_u\alpha_u-r_s\alpha_s)]$$
$$4r_u(1+\tilde{\alpha}_s(e_s))dt\wedge \hat{e}_s \wedge \hat{e}_u \wedge \hat{X},$$
where $e_s\in E^{ss}$ and $\alpha_s(e_s)=1$.

Since $r_u(1+\tilde{\alpha}_s(e_s))>0$, for sufficiently small $\epsilon$, the resulting $(\alpha_-,\alpha_+)$ has the desired properties.

\end{proof}



\section{Uniqueness Of The Underlying (Bi-)Contact Structures}\label{5}

In this section, we want to establish various uniqueness theorems, about the (bi)-contact structures underlying a given (projectively) Anosov flow. Let $X$ be the $C^1$ vector field generating such flow and  $\xi$ be any oriented plane field such that $X\subset \xi$. In particular, we want to establish the uniqueness, up to bi-contact homotopy (see Definition~\ref{bilocal}), of the supporting bi-contact structure, as well as explore the conditions under which, we can retrieve the information of such bi-contact structure, from only one of the contact structures. First, we need a definition.

\begin{definition}
We call a vector $v\in T_pM$ {\em dynamically positive (negative)}, if the plane $\langle v \rangle \oplus \langle X \rangle$ can be extended to a positive (negative) contact structure $\xi$, such that for some $\xi_-$ ($\xi_+$), there exists a supporting bi-contact structure $(\xi_-,\xi)$ ($(\xi,\xi_+)$) for $X$. We call a vector field $v$ {\em dynamically positive (negative)} on the set $U \subset M$, if it is dynamically positive (negative) at every $p \in U$. Finally, we call a plane field $\xi$ {\em dynamically positive (negative)} on the set $U \subset M$, if $\xi=\langle v \rangle \oplus \langle X \rangle$ for some dynamically positive (negative) vector field $v$ on $U$.
\end{definition}

This is basically a bi-contact geometric way of saying that a vector (or vector field or a plane field) is dynamically positive (or negative) at a point, if it lies in the interior of the first or third quadrant (the second or forth quadrant) of Figure~1~(b). In the above definition, we avoid referring to the invariant bundles and give a purely bi-contact geometric definition, in which extensions of the vector is needed. However, the defined conditions are in fact point wise, as one can codify the long term behavior in terms of the invariant bundles and being dynamically negative or positive can then be interpreted as being in the corresponding quadrants with respect to the splitting into invariant bundles. 

Note that if $\xi_+$ is a positive contact structure coming from a supporting bi-contact structure $(\xi_-,\xi_+)$, by Remark~\ref{geomca}, $\xi_+$ is dynamically positive everywhere. But this is not true in general. That is, a general supporting positive contact structure can be dynamically negative on a subset of the manifold. However, as we will shortly discuss, the behavior of the contact structure can be easily understood in such regions.

In particular, note that if $(\xi_-,\xi_+)$ is a supporting bi-contact structure for $X$, then $\xi_+$ ($\xi_-$) is dynamically positive (negative) on $M$.

Next, we see that when a supporting positive (negative) contact structure is dynamically positive (negative) everywhere on $M$, it is in fact isotopic, through supporting positive (negative) contact structures, to a positive (negative) contact structure, coming from any given supporting bi-contact structure. In particular, such contact structure is part of a supporting bi-contact structure.

 \begin{lemma}\label{linlem}
 Let $(\xi_-,\xi_+)$ be a supporting bi-contact structure for the projectively Anosov flow, generated by $X$, and $\xi$ any supporting positive contact structure, which is dynamically positive everywhere. Then, $\xi$ is isotopic to $\xi_+$, through supporting positive contact structures, which are dynamically positive everywhere. 
 \end{lemma}

\begin{proof}
It suffices to show that linear interpolation of $\xi$ and $\xi_+$ is through positive contact structures and Gray's theorem guarantees the existence of isotopy. For simplicity, we assume $\pi^{-1}(E^s)$ and $\pi^{-1}(E^u)$ are $C^1$ plane fields. Otherwise, we can use the approximations used in the proof of Theorem~\ref{main} and the fact that both projective Anosovity and contactness are open conditions.

Choose $C^1$ 1-forms $\alpha_s$ and $\alpha_u$ such that $\ker{\alpha_s}=\pi^{-1}(E^u)$, $\ker{\alpha_u}=\pi^{-1}(E^s)$ and $\xi_+=\ker{\alpha_+}$, where $$\alpha_+:=\frac{\alpha_u-\alpha_s}{2}.$$

Then, there exists $C^1$ function $f$, such that $$\xi=\ker{\frac{f\alpha_u-\alpha_s}{2}}.$$

Letting $\alpha_+':=\frac{f\alpha_u-\alpha_S}{2}$, we show that for all $t\in [0,1]$,
$$\alpha_t:=(1-t)\alpha_++t\alpha_+'$$
is a positive contact structure.

In the following, we show that $\alpha_t$ s are positive contact forms for all $t\in [0,1]$ and appealing to Gray's theorem provides the desired isotopy, which can also be explicitly computed to be of the form $x\mapsto \phi^{\tau(x)}(x)$, where $\tau{(x)}$ depends on the point $x$ and the expansion rates $r_u$ and $r_s$.

Choose some transverse plane field $\eta$ and assume $e_s\in \pi^{-1}(E^s)\cap\eta$ and $e_u\in \pi^{-1}(E^u)\cap\eta$ are the vector fields defined by $\alpha_s(e_s)=\alpha_u(e_u)=1$, and $r_s$ and $r_u$ are the corresponding expansion rates of stable and unstable directions, respectively, i.e.
$$-\mathcal{L}_Xe_s=r_se_s -q_s^\eta \ X\ \text{ and }-\mathcal{L}_Xe_u=r_ue_u -q_s^\eta \ X,$$
for some real functions $q_s^\eta,q_u^\eta$ (see Proposition~\ref{cageoformula2}).

We can easily compute (as in proof of Theorem~\ref{main} and using Proposition~\ref{cageoformula2}):
$$4(\alpha_0\wedge d\alpha_0)(e_s,e_u,X)=\big( \alpha_u\wedge d\alpha_u -\alpha_u \wedge d\alpha_s-\alpha_s\wedge d\alpha_u +\alpha_s\wedge d\alpha_s\big)(e_s,e_u,X)$$
$$=-\alpha_u(e_u)\alpha_s([e_s,X])+\alpha_s(e_s)\alpha_u([e_u,X])=r_u-r_s>0;$$
$$4(\alpha_1\wedge d\alpha_1)(e_s,e_u,X)=\big(f \alpha_u\wedge d(f\alpha_u) -f\alpha_u \wedge d\alpha_s-\alpha_s\wedge d(f\alpha_u) +\alpha_s\wedge d\alpha_s\big)(e_s,e_u,X)$$
$$=-f\alpha_u(e_u)\alpha_s([e_s,X])+\alpha_s(e_s)\left[X\cdot (f\alpha_u(e_u))+f\alpha_u([e_u,X])\right]=fr_u-fr_s+X\cdot f>0;$$
$$4(\alpha_0\wedge d\alpha_1)(e_s,e_u,X)=\big( \alpha_u\wedge d(f\alpha_u) -\alpha_u \wedge d\alpha_s-\alpha_s\wedge d(f\alpha_u) +\alpha_s\wedge d\alpha_s\big)(e_s,e_u,X)$$
$$=-\alpha_u(e_u)\alpha_s([e_s,X])+\alpha_s(e_s)\left[X\cdot (f\alpha_u(e_u))+f\alpha_u([e_u,X])\right]=fr_u-r_s+X\cdot f;$$
$$4(\alpha_1\wedge d\alpha_0)(e_s,e_u,X)=\big( f\alpha_u\wedge d\alpha_u -f\alpha_u \wedge d\alpha_s-\alpha_s\wedge d\alpha_u +\alpha_s\wedge d\alpha_s\big)(e_s,e_u,X)$$
$$=-f\alpha_u(e_u)\alpha_s([e_s,X])+\alpha_s(e_s)\alpha_u([e_u,X])=r_u-fr_s.$$

It yields
$$(\alpha_t\wedge d\alpha_t)(e_s,e_u,X)=t^2(fr_u-fr_s+X\cdot f)+(1-t)^2(r_u-r_s)+t(1-t)(r_u-r_s+fr_u-fr_s+X\cdot f)>0,$$
completing the proof.
\end{proof}

This, in particular, implies that the supporting bi-contact structure for any projectively Anosov flow is unique, up to homotopy through supporting bi-contact structures.

\begin{theorem}\label{uniquebi}
If $(\xi_-,\xi_+)$ and $(\xi_-',\xi_+')$ are two supporting bi-contact structures for a projectively Anosov flow, generated by $X$, then they are homotopic through supporting bi-contact structures.
\end{theorem}

\begin{proof}
By Remark~\ref{geomca}, $\xi_+'$ and $\xi_-'$ are dynamically positive and negative everywhere, respectively. The proof of Lemma~\ref{linlem} finishes the proof (See Figure~3).
\end{proof}

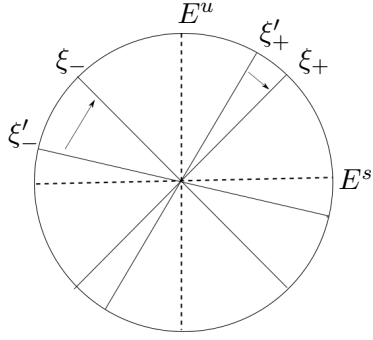
\begin{figure}\label{uniq}
  \center \begin{overpic}[width=4cm]{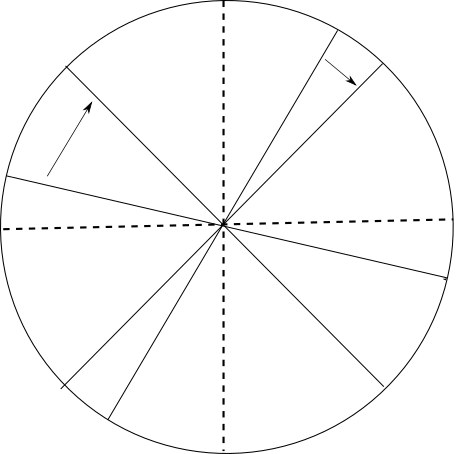}
  \put(8,100){$\xi_-$}
  \put(100,100){$\xi_+$}
    \put(-10,72){$\xi_-'$}
  \put(85. ,110){$\xi_+'$}
    \put(54,117){$E^u$}
      \put(115,54){$E^s$}

  \end{overpic}
\caption{Uniqueness of the supporting bi-contact structure}
\end{figure}

It is important to understand how a positive contact structure $\xi$ with projectively Anosov vector field $X\subset \xi$ behaves in a region, where it is dynamically negative (similarly, we can describe the behavior of a negative contact structure in a region, where it is dynamically positive). Consider a transverse plane field $\eta$, which is $C^1$ in $X$-direction, and using any Riemannian metric as described above, define the function
$$\theta_\xi :M\rightarrow [0,2\pi),$$
which measures the angle between $\xi\cap \eta$ and the bi-sector of $\eta\cap\pi^{-1}(E^s)$ and $\eta\cap\pi^{-1}(E^u)$ in the positive region. Note that this function is continuous and differentiable with respect to $X$, where $\xi$ is dynamically negative, $\xi=\pi^{-1}(E^s)$ or $\xi=\pi^{-1}(E^u)$. Remark~\ref{geomca} guarantees that at such points $$X\cdot\theta_\xi<0,$$
since at those points, the flow rotates $\xi$ clockwise in those regions (see Figure~1~(b)) and by Frobenius theorem, $\xi$ needs rotate faster in a clockwise fashion, to stay a positive contact structure.

Now consider the family of plane fields 
$$\eta_\theta:=\langle X \rangle \oplus l_\theta,$$
for $\theta \in I_-:=[\frac{\pi}{4},\frac{3\pi}{4}]\cup [\frac{5\pi}{4},\frac{7\pi}{4}]$, where $l_\theta\subset\eta$ is the oriented line field which has angle $\theta$ with the dynamically positive bi-sector of $\eta\cap\pi^{-1}(E^s)$ and $\eta\cap\pi^{-1}(E^u)$. Note that such $l_\theta$ is either dynamically negative, or the same as $E^s$ or $E^u$ (ignoring the orientation). After a generic smooth perturbation of $\xi$, we can assume the set $\Sigma_\theta:=\{x\in M \text{ s.t. } \xi=\eta_\theta \}$ is a differentiable manifold, which is transverse to $X$, since $X\cdot\theta_\xi<0$ (and using {\em the implicit function theorem}). Hence, such solution set is a union of tori, since the splitting $TM/\langle x\rangle \simeq E^s\oplus E^u$ would trivialize the tangent space of such surface. Therefore, if $N\subset M$ is the set on which $\xi$ is dynamically negative, then
$$\bar{N}=\Sigma:=\bigcup\limits_{\theta\in I_-} \Sigma_\theta \simeq \bigcup\limits_{1\leq i \leq k}T_i\times [0,1],$$
for some integer $k$, where $T_i$ s are tori and for each $1\leq i \leq k$ and $\tau \in [0,1]$, $X$ is transverse to $T_i\times \{\tau\}$.

From the above observations and what we know about Anosov flows, we can derive a host of uniqueness theorems about $\xi$.

\begin{lemma}\label{dynneg}
Using the above notations, let $X$ be an Anosov flows and $N\subset M$, the subset of $M$ on which the positive contact structure $\xi$ is dynamically negative. 

a) $\bar{N} \simeq \bigcup\limits_{1\leq i \leq k}T_i\times [0,1]$, where $T_i$ s are incompressible tori;

b) $\partial\bar{N}=\{x\in M \text{ s.t. }\xi=\pi^{-1}(E^s) \} \cup \{x\in M \text{ s.t. }\xi=\pi^{-1}(E^u) \}$;

c) If $T_1$, $T_2$, ..., $T_j$ of part (a) are parallel through transverse tori, then there exists a map $$(\mathbb{S}^1\times \mathbb{S}^1\times [0,(j-1)\pi]\text{ with coordinates }(s,t,\theta), \ker{\{ \cos{\theta}\ dt+\sin{\theta}\ ds \}} ) \rightarrow (M,\xi),$$
which is a contact embedding on $(\mathbb{S}^1\times \mathbb{S}^1\times [0,(j-1)\pi))$.

d) If we assume $X$ to be only projectively Anosov (not necessarily Anosov), we can conclude all the above, except $T_i$ might not be incompressible.

\end{lemma}

\begin{proof}
Part (a) and (b) follow from the above discussion and the fact that any surface which is transverse to an Anosov flow is an incompressible torus \cite{brunella,fenley2,mosher}. For Part (c), notice that if we consider two immediate tori, we have a half-twist of the flow (a Giroux $\pi$-torsion) in between (see Remark~\ref{torsion}). More precisely, we can reparametrize the angle $\theta$ of the above discussion, by the flowlines, when in the region between any two adjacent tori, where the flow is dynamically positive (and where the flow is dynamically negative, we automatically have $X\cdot\theta_\xi<0$). Therefore, we get a contact embedding of $$\left( [0,\pi]\times \mathbb{S}^1\times\mathbb{S}^1 \text{ with coordinates }(t,\phi_1,\phi_2),\ker{\{\cos{t}\ d\phi_1+sin{t}\ d\phi_2} \right\})\rightarrow(M,\xi).$$
\end{proof}

\begin{theorem}
If $M$ is atoroidal and $(\xi_-,\xi_+)$ a supporting bi-contact structure for the Anosov vector field $X$ on $M$, then for any supporting positive contact structure $\xi$, $\xi$ is isotopic to $\xi_+$, through supporting contact structures.
\end{theorem}

\begin{proof}
By Lemma~\ref{dynneg} $\xi$ is dynamically positive everywhere, and Lemma~\ref{linlem} finishes the proof.
\end{proof}

An Anosov flow is called {\em $\mathbb{R}$-covered}, if both lifts of its stable and unstable foliations to the universal cover are equivalent to the the trivial foliation of $\mathbb{R}^3$ by planes . This is an important class of Anosov flows and is studied in depth, in the works of Fenley, Bartbot, as well as many others. In particular, it is known \cite{barbot,fenley1} that when an {\em $\mathbb{R}$-covered} Anosov flow is {\em skewed}, i.e. it is not orbit equivalent to a suspension Anosov flow, then, it does not admit any transverse embedded surface (see Theorem~3.23 of \cite{topanosov}). Hence,

\begin{theorem}
Let $X$ be a skewed $\mathbb{R}$-covered Anosov vector field, supported by the bi-contact structure $(\xi_-,\xi_+)$ on $M$, and let $\xi$ be any supporting positive contact structure. Then, $\xi$ is isotopic to $\xi_+$, through supporting contact structures.
\end{theorem}

In the case of $M$ being a torus bundle, the underlying contact structures can be characterized by having the minimum torsion (see Remark~\ref{torsion}). Although, similar phenomena can be observed in the case of projective Anosov flows, we state the theorem for Anosov flows, for which the relation of torsion and symplectic fillability is established in \cite{gay,mnw}. The proof relies on the classification of contact structures on torus bundles and $\mathbb{T}^2\times I$ by Ko Honda and one should consult \cite{honda1} and \cite{honda2} for more details and precise definitions.

\begin{theorem}\label{john}
Let $X$ be the suspension of an Anosov diffeomorphism of torus, supported by the bi-contact structure $(\xi_-,\xi_+)$, and $\xi$ a positive supporting contact structure. Then, $\xi$ is isotopic through supporting bi-contact structures to $\xi_+$, if and only if, $\xi$ is strongly symplectically fillable.
\end{theorem}

\begin{proof}


If $\xi$ is dynamically positive everywhere, Lemma~\ref{linlem} yields the isotopy. Otherwise, for any incompressible torus $T_i$ in Lemma~\ref{dynneg}, the flow is a suspension flow for an appropriate Anosov diffeomorphism of $T_i$. Notice that there is at least two of such $T_i$, since $\xi$ is coorientable. The idea is that, in this case, $\xi$ rotates at least $2\pi$ more than $\xi_+$, as we move in $\mathbb{S}^1$-direction (see Lemma~\ref{dynneg}) and since $\xi_+$ rotates some itself, that means that $\xi$ rotates more than $2\pi$. Therefore, $\xi$ contains Giroux torsion and is not symplectically fillable.

Let $\tilde{M}:=\overline{M\backslash T_1}\simeq \mathbb{T}^2\times I$, where we have compactified $M\backslash T_1$, by gluing two copies of $T_1$ along the boundary. i.e. $T^1_1$ and $T_1^2$, such that $\partial \tilde{M}=-T^1_1 \sqcup T_1^2$ (abusing notation, we call the induced contact structures, $\xi_+$ and $\xi$). After a choice of basis for $\mathbb{T}^2$, let $s_+^i$ ($s^i$), $i=1,2$, be the slope of the {\em characteristic foliation} of $\xi_+$ ($\xi$) on $T_1^i$, respectively. That is the foliation of $T_1^i$ by $TT_1^i\cap \xi_+$ ($TT_1^i\cap \xi$).

Note that since $\xi_+$ is universally tight, by \cite{honda2}, $\xi_+$ has {\em nonnegative twisting} as it goes from $T_1^1$ to $T_1^2$. Furthermore, since $\xi_+$ does not contain Giroux torsion, such twisting is less than $2\pi$. We claim that $s_+^1\neq s_+^2$ and $s^1\neq s^2$. That is because an Anosov diffeomorphism of torus preserves exactly two slopes of the torus and those are the intersections of $\pi^{-1}(E^s)$ and $\pi^{-1}(E^u)$ with the boundary.

Now by Lemma~\ref{dynneg}~c), there exist at least a $\pi$-twisting between $T_1^1$ and $T_2$, as well as between $T_2$ and $T_1^2$. That is a total of at least $2\pi$-twisting. i.e. a contact embedding

$$\left( [0,2\pi]\times \mathbb{S}^1\times\mathbb{S}^1 \text{ with coordinates }(t,\phi_1,\phi_2),\ker{\{ \cos{t}\ d\phi_1+sin{t}\ d\phi_2\}} \right)\rightarrow(\tilde{M},\xi),$$
with $\{0\} \times \mathbb{S}^1\times\mathbb{S}^1\rightarrow T_1^1$. Since $Im(\{1\} \times \mathbb{S}^1\times\mathbb{S}^1)$ has the same slope $s_1$ as $T_1^1$ (after a $2\pi$-twist), this implies $Im(\{1\} \times \mathbb{S}^1\times\mathbb{S}^1) \cap T^2_1=\emptyset$ and therefore, we will achieve an embedding
$$\left( [0,2\pi]\times \mathbb{S}^1\times\mathbb{S}^1 \text{ with coordinates }(t,\phi_1,\phi_2),\ker{\{\cos{t}\ d\phi_1+sin{t}\ d\phi_2}\} \right)\rightarrow(M,\xi),$$
meaning that $\xi$ contains Giroux torsion.
\end{proof}


\section{A Characterization of Anosovity Based on Reeb Flows and\\ Consequences}\label{6}

In this section, we use ideas developed in Section~\ref{3} and the proof of Theorem~\ref{main} to give a characterization of Anosovity, based on the {\em Reeb flows}, associated to the underlying contact structures of a projectively Anosov, as well as discuss its contact topological implications in the Anosov case. First, recall:

\begin{definition}
Given a contact manifold $(M,\xi)$, for any choice of contact form $\alpha$ for $\xi$, there exists a unique vector field $R_\alpha$, named the {\em Reeb vector field} with the following properties:
$$1)\text{ } \alpha(R_\alpha)=1,$$
$$2) \text{ }d\alpha(R_\alpha,.)=0.$$
\end{definition}

\begin{example}
1) For the standard tight contact structure on $\mathbb{S}^3$, described in Example~\ref{excontact}~2), the Reeb vector field is tangent to Hopf fibration, for an appropriate choice of contact form.

2) For all the contact structures of Example~\ref{excontact}~3) on $\mathbb{T}^3$, the unit orthogonal vector field (considering the flat metric on $\mathbb{T}\simeq\mathbb{R}^3$), is the Reeb vector field.

3) The geodesic flow on the unit tangent space of a hyperbolic surface is the Reeb vector field for the tautological 1-form (which is a contact form). In this case the Reeb vector field itself is Anosov.
\end{example}

\begin{theorem}\label{reebchar}
Let $X$ be a projectively Anosov vector field on $M$. Then, the followings are equivalent:

(1) $X$ is Anosov;

(2) There exists a supporting bi-contact structure $(\xi_-,\xi_+)$, such that $\xi_+$ admits a Reeb vector field, which is dynamically negative everywhere;

(3) There exists a supporting bi-contact structure $(\xi_-,\xi_+)$, such that $\xi_-$ admits a Reeb vector field, which is dynamically positive everywhere.
\end{theorem}

\begin{proof}
For simplicity assume $\pi^{-1}(E^s)$ and $\pi^{-1}(E^u)$ are $C^1$ plane fields. The general case follows from the approximations described in the proof of Theorem~\ref{main} and the fact that Anosovity, as well as being dynamically positive (negative) everywhere, are open conditions.

Assuming $(1)$, we now show $(2)$.

Choose a transverse plane field $\eta$ and let $e_s\in \pi^{-1}(E^s)\cap\eta$ and $e_u\in \pi^{-1}(E^u)\cap\eta$ be the unit vector fields with respect to the Riemannian metric satisfying $r_s<0<r_u$, and $\alpha_s$ defined by $\alpha_s(e_s)=1$ and $\alpha_s(\pi^{-1}(E^u))=0$ (see proof of Theorem~\ref{main} for notation). Similarly, define $\alpha_u$. Define $\alpha_+:=\frac{1}{2}(\alpha_u-\alpha_s)$. Note that $(\xi_-,\xi_+:=\ker{\alpha_+})$ is a supporting bi-contact structure, for an appropriate choice of $\xi_-$. The span of the Reeb vector field, $R_{\alpha_+}$, is determined by the two equations
$$d\alpha_+(X,R_{\alpha_+})=0=d\alpha_+(e_+,R_{\alpha_+}),$$
where $e_+\in\xi_+$ is a vector field such that that $\langle X \rangle \oplus \langle e_+ \rangle=\xi_+$.

Consider the vector $v:=-r_se_u-r_ue_s$ and note that since $r_s<0<r_u$, such vector is dynamically negative. Compute
$$d\alpha_+(X,v)=-r_sd\alpha_+(X,e_u)-r_ud\alpha_+(X,e_s)=-r_sr_u+r_sr_u=0.$$

This implies $R_{\alpha_+}\subset \langle X,v \rangle$ and therefore, $R_{\alpha_+}$ is dynamically negative everywhere.

Now assume $(2)$ and we establish $(1)$.

Let $\alpha_+$ be such contact form for $\xi_+$. Define $\alpha_u$ and $\alpha_s$ such that $\alpha_+=\frac{1}{2}(\alpha_u-\alpha_s)$ and $\alpha_s(\pi^{-1}(E^u))=\alpha_u(\pi^{-1}(E^s))=0$. Finally, choose a transverse plane field $\eta$ and define the Riemannian metric such that for unit vectors $e_s\in \pi^{-1}(E^s)\cap\eta$ and $e_u\in \pi^{-1}(E^u)\cap\eta$, we have $\alpha_s(e_s)=\alpha_u(e_u)=1$. By the above computation, we observe $R_{\alpha_+}\subset \langle X,-r_se_u-r_ue_s \rangle$. Since such vector is dynamically negative, this implies $$r_s<0<r_u,$$
and therefore, $X$ is Anosov.

Equivalence of $(1)$ and $(3)$ is similar.
\end{proof}

The above characterization can be used to show that the contact structures, underlying an Anosov flow, are {\em hypertight}. That is, they admit contact forms, whose associated Reeb flows do not have any contractible periodic orbits. The importance of hypertightness is due to the celebrated works of Helmut Hofer, et al in Reeb dynamics (see \cite{hofermain,hwz,wendl}) which show that if $(M,\xi)$ is a hypertight contact manifold, then, $\xi$ is tight, $M$ is irreducible and $(M,\xi)$ does not admit an exact symplectic cobordism to $(\mathbb{S}^3,\xi_{std})$. Note that in our case, such properties hold for any covering of $(M,\xi)$ as well, since we can lift the Anosov flow on $M$ to an Anosov flow on the covering.

\begin{theorem}
Let $(\xi_-,\xi_+)$ be a supporting bi-contact structure for an Anosov flow. Then, $\xi_-$ and $\xi_+$ are hypertight.
\end{theorem}

\begin{proof}
Consider $\xi_+$ (the $\xi_-$ case is similar). By Theorem~\ref{reebchar}, $\xi_+$ admits a Reeb vector field, which is dynamically negative everywhere. In particular, it is transverse to both stable and unstable foliations. However, it is a well known fact that there are no contractible closed transversals for stable or unstable foliations, associated to an Anosov flow (see Lemma~3.1 in \cite{topanosov}). More precisely, by \cite{haef}, if a codimension 1 foliation on a 3-manifold admits a contractible closed transversal, then there exists a closed loop in one of the leaves of the foliation, whose holonomy is trivial from one side and non-trivial from the other. That is impossible for the stable (unstable) foliation of an Anosov flow, since such holonomy needs to be contracting (expanding) on both sides.
\end{proof}

\begin{corollary}
Let $(\xi_-,\xi_+)$ be a supporting bi-contact structure for an Anosov flow. Then,

(1) $\xi_-$ and $\xi_+$ are universally tight;

(2) $M$ is irreducible;

(3) there are no exact symplectic cobordisms from $(M,\xi_+)$ or $(-M,\xi_-)$ to $(\mathbb{S}^3,\xi_{std})$.
\end{corollary}
\section{Bi-Contact Topology And Anosov Dynamics: Remarks And Questions}\label{7}

In Theorem~\ref{main}, we proved that the Anosovity of a flow is equivalent to a host of contact and symplectic geometric conditions. This bridge naturally creates a hierarchy of geometric conditions on the flow and therefore, a new filtration of Anosov dynamics, starting with projectively Anosov flows and ending with Anosov flows. It is of general interest to understand which layer of geometric conditions is responsible for properties of Anosov flows and introduces new geometric and topological tools to study questions in Anosov dynamics. In this section, we want to establish such hierarchy, make some remarks and formalize a platform for such study. Noting that such contact topological properties are preserved under homotopy of a projectively Anosov flow, this study also sheds light on the classical problem of classifying Anosov flows, by exploring the dynamical phenomena which can appear under {\em bi-contact homotopy}.

In Theorem~\ref{main}, we observed that underlying any Anosov flow is a bi-contact structure, corresponding to the projective Anosovity of the flow. In order to reduce the questions about Anosov dynamics to contact topological questions, we first need to understand the dependence of Anosovity on the geometry of the supporting bi-contact structure. First, we define two notions of equivalence for bi-contact structure, which can describe deformation of a projectively Anosov flow.

\begin{definition}\label{bilocal}
We call two bi-contact structures $(\xi_-,\xi_+)$ and $(\xi_-',\xi_+')$ {\em bi-contact homotopic}, if there exists a homotopy of bi-contact structures $(\xi_-^t,\xi_+^t),t\in[0,1]$ with $(\xi_-^0,\xi_+^0)=(\xi_-,\xi_+)$ and $(\xi_-^1,\xi_+^1)=(\xi_-',\xi_+')$. We call the two bi-contact structures {\em isotopic}, if such homotopy is induced by an isotopy of the underlying manifold.
\end{definition}

Note that bi-contact homotopy of $(\xi_-,\xi_+)$ and $(\xi_-',\xi_+')$ is equivalent to the supported projectively Anosov flows to be homotopic through projectively Anosov flows. Also in this case, by Gray's theorem, $\xi_-$ and $\xi_-'$, as well as $\xi_+$ and $\xi_+'$, are isotopic, but not necessary through the same isotopy.

The notion of bi-contact isotopy seems to be too rigid for the study of many geometric and topological aspects of Anosov dynamics, since it is not even preserved under general perturbations of a (projectively) Anosov flow. Note that even for a fixed projectively Anosov flow, the supporting bi-contact structure is not a priori unique up to isotopy (See Theorem~\ref{uniquebi}). On the other hand, bi-contact homotopies are more natural for many problems of topological and geometric nature. That is partly due to the {\em structural stability} of Anosov flows. That is, the perturbation of an Anosov flow is an Anosov flow which is orbit equivalent to the original flow. It is also known that the same holds for a generic projectively Anosov flow \cite{clinic}. Moreover, Theorem~\ref{uniquebi} shows that for a fixed projectively Anosov flow, the supporting bi-contact structure is unique up to bi-contact homotopy. 

However, the dependence of Anosovity on bi-contact homotopy is yet to be understood.

\begin{question}
Let $(\xi_-,\xi_+)$ be a bi-contact structure, supporting an Anosov flow, and $(\xi_-',\xi_+')$ another bi-contact structure which is bi-contact homotopic to $(\xi_-',\xi_+')$. Is a projectively Anosov flow which is supported by $(\xi_-',\xi_+')$ Anosov?
\end{question}

While an affirmative answer to the above question might be too optimistic (although we are not aware of an explicit counterexample), confirming the following more modest conjecture can still reduce many problems in Anosov dynamics, to a great extent, to contact topological problems. That includes problems concerning the orbit structures, their periodic orbits and the classification of Anosov flows up to orbit equivalence.

\begin{conjecture}
Two Anosov flows which are supported by bi-contact homotopic bi-contact structures are orbit equivalent. Equivalently, two Anosov flows which are homotopic through projectively Anosov flows are orbit equivalent.
\end{conjecture}

A weaker notion than bi-contact homotopy, is when given two bi-contact structures, the positive contact structures, as well as the negative contact structures, are isotopic. But the transversality of the two might be violated during the homotopy. In other words, can a pair of positive and negative contact structures be transverse in two distinct ways?

\begin{question}
Let $(\xi_-,\xi_+)$ and $(\xi_-',\xi_+')$ be two bi-contact structures, such that $\xi_-$ and $\xi_-'$, as well as $\xi_+$ and $\xi_+'$, are isotopic. Are $(\xi_-,\xi_+)$ and $(\xi_-',\xi_+')$ bi-contact homotopic?
\end{question}

After questions regarding the relation of Anosovity and the geometry of the supporting bi-contact structure, we can ask about the relation to the topology of bi-contact structures. More precisely, Anosovity implies rigid contact and symplectic topological properties of the underlying contact structures (see Theorem~\ref{main} and Remark~\ref{topprop}). It is natural to ask about the degree to which these topological properties are responsible for the dynamical properties of Anosov flows.

\begin{definition}
We call a bi-contact $(\xi_-,\xi_+)$ structure {\em tight}, if both $\xi_-$ and $\xi_+$ are tight, and we call a projectively Anosov flow {\em tight}, if it is supported by a tight bi-contact structure.
\end{definition}

To the best of our knowledge, it is not known whether the tightness of one of the supporting contact structures imply the same property for the other.

\begin{question}
If $(\xi_-,\xi_+)$ is a bi-contact structure, such that $\xi_-$ is tight. Can $\xi_+$ be overtwisted?
\end{question}

In \cite{mitsumatsu}, Mitsumatsu introduces a criteria of making a pair positive and negative contact structures transverse, which yields tight bi-contact structures on $\mathbb{T}^3$ and nil-manifolds (note that these projectively Anosov flows are not Anosov \cite{fun}). Here, we put down the explicit examples on $\mathbb{T}^3$.

\begin{example}\label{exbicontact}
After an isotopy, and for integers $m,n>0$, we consider the positive and negative tight contact structures of Example~\ref{excontact}~3) on $\mathbb{T}^3$, $\xi_n=\ker{dz+\epsilon\{\cos{2\pi nz}dx-\sin{2\pi nz}dy\}}$ and $\xi_{-m}=\ker{dz+\epsilon'\{\cos{2\pi mz}dx+\sin{2\pi mz}dy\}}$, respectively. It is easy to observe that if $\epsilon\neq\epsilon'$, then $\xi_n$ and $\xi_{-m}$ are transverse everywhere, and therefore, their intersection contains tight projectively Anosov vector fields.
\end{example}



\begin{definition}
We call a bi-contact structure $(\xi_-,\xi_+)$ on $M$ {\em weakly, strongly or exactly symplectically bi-fillable}, if there exists $(X,\omega)$, which is a weak, strong or exact symplectically filling for $(M,\xi_+)\sqcup(-M,\xi_-)$, respectively, where $-M$ is $M$ with reversed orientation. We call a projectively Anosov flow supported by such bi-contact structure, {\em weakly, strongly or exactly symplectically bi-fillable}, respectively. Furthermore, we call the symplectic bi-filling {\em trivial}, if $X\simeq M\times[0,1]$.
\end{definition}

Note that any Anosov flow is exactly symplectically bi-fillable. Furthermore, any exactly bi-fillable projectively Anosov flow is strongly bi-fillable and any strongly bi-fillable projectively Anosov flow is weakly bi- fillable.

Using the idea of Example~\ref{exfilling} 2), we can show that these projectively Anosov flows are in fact, weakly bi-fillable.

\begin{theorem}
The tight projectively Anosov flows of Example~\ref{exbicontact} are trivially weakly symplectically bi-fillable. 
\end{theorem}

\begin{proof}
Let $X=\mathbb{T}^2\times A$, where $A$ is an annulus. Consider the coordinates $(x,y)$ for $\mathbb{T}^2$ and let $z$ be the angular coordinate of $A$, near its boundary. If $\omega_1$ and $\omega_2$ are some area forms on $\mathbb{T}^2$ and $A$, respectively, then $\omega=\omega_1\oplus\omega_2$ will be a symplectic form on $X$, such that $\omega|_{\ker{dz}}>0$. Choosing small enough $\epsilon,\epsilon'>0$ in Example~\ref{exbicontact}, $\xi_n$ and $\xi_{-m}$ would be arbitrary close to $\ker{dz}$ and therefore, $\omega|_{\xi_n},\omega|_{\xi_{-m}}>0$, implying that $(\xi_{-m},\xi_n)$ is weakly symplectically bi-fillable, for any pair of integers $m,n>0$.
\end{proof}

Using \cite{elfilling,etfilling}, we know such tight projectively Anosov flows are not strongly symplectically bi-fillable (that would conclude $\xi_n$ and $\xi_{-m}$ to be strongly symplectically bi-fillable, which is not the case \cite{eltorus}).

\begin{corollary}
There are (trivially) weakly symplectically bi-fillable projectively Anosov flows, which are not strongly symplectically bi-fillable.
\end{corollary}

The properness of other inclusions in the described hierarchy of projectively Anosov flows remains an open problem.

\begin{question}\label{qprop}
Are there tight projectively Anosov flows, which are not weakly symplectically bi-fillable? Are there strongly bi-fillable projectively Anosov flows, which are not exactly bi-fillable? Are there exactly bi-fillable projectively Anosov flows, which are not Anosov?
\end{question}

\begin{remark}
Here, we remark that our filtration of contact and symplectic conditions on a projectively Anosov flow is what we found more natural and can be refined and modified in other ways and using other conditions, for instance on the topology of the symplectic fillings, etc. In particular, we note that in our definition of symplectic bi-fillings for a projectively Anosov flow, supported by a bi-contact structure $(\xi_-,\xi_+)$, we did not consider bi-fillability for both $(\xi_-,\xi_+)$ and $(-\xi_-,\xi_+)$, a condition which is satisfied for Anosov flows. Note that symplectic fillability, for a contact manifold $(M,\xi)$ with connected boundary, does not depend on the coorientation of the contact structure, since if $(X,\omega)$ is a symplectic filling for such contact manifold, then $(X,-\omega)$ would be a symplectic filling for $(M,-\xi)$. But when we have disconnected boundary, like in the case of bi-contact structures, fillability properties might change if we flip the orientation of only one of the contact structures.
\end{remark}

\Addresses
\end{document}